\documentclass{amsart}

\usepackage{amssymb, array, verbatim, amscd, color}
\usepackage{amsmath, amsfonts,amsthm}
\usepackage{enumerate}
\usepackage[all]{xy}
\usepackage{xy}
\usepackage{hyperref}

\theoremstyle{plain}

\newtheorem{theorem}{Theorem}[subsection]
\newtheorem{corollary}[theorem]{Corollary}
\newtheorem{lemma}[theorem]{Lemma}
\newtheorem{proposition}[theorem]{Proposition}

\theoremstyle{definition}

\newtheorem{definition}[theorem]{Definition}
\newtheorem{example}[theorem]{Example}

\newtheorem{remark}[theorem]{Remark}

\newtheorem{setup}[theorem]{Setup}

\newcommand{\CA}{{\mathcal A}}
\newcommand{\CB}{{\mathcal B}}
\newcommand{\CC}{{\mathcal C}}

\newcommand{\CE}{{\mathcal E}}

\newcommand{\CI}{{\mathcal I}}
\newcommand{\CJ}{{\mathcal J}}
\newcommand{\CK}{{\mathcal K}}
\newcommand{\CL}{{\mathcal L}}
\newcommand{\CM}{{\mathcal M}}
\newcommand{\CP}{{\mathcal P}}
\newcommand{\CQ}{{\mathcal Q}}

\newcommand{\CT}{{\mathcal{T}}}
\newcommand{\CV}{{\mathcal V}}

\newcommand{\CX}{{\mathcal{X}}}

\newcommand{\Q}{{\mathbb Q}}
\newcommand{\Z}{{\mathbb Z}}

\newcommand{\X}{{\mathcal X}}
\newcommand{\Y}{{\mathcal Y}}

\newcommand{\Ker}{\mathrm{Ker}}

\newcommand{\Ob}{\mathrm{Ob}}

\newcommand{\aadd}{\mathrm{add}}

\newcommand{\Add}{\mathrm{Add}}

\newcommand{\frakF}{\mathfrak{F}}
\newcommand{\frakE}{\mathfrak{E}}
\newcommand{\frakD}{\mathfrak{D}}
\newcommand{\frakG}{\mathfrak{G}}
\newcommand{\frakH}{\mathfrak{H}}

\newcommand{\Modr}{\mathrm{Mod}\text{-}}

\newcommand{\Hom}{\operatorname{Hom}}

\newcommand{\im}{\operatorname{im}}

\newcommand{\Ext}{\operatorname{Ext}}

\newcommand{\wrt}{with respect to }
\newcommand{\PEF}{{\Phi}_\frakE(\frakF)}
\newcommand{\CPEF}{{\Phi}^\frakE(\frakF)}
\newcommand{\CPEG}{{\Phi}^\frakE(\mathfrak{G})}

\newcommand{\Ph}{{{\Phi}}}
\newcommand{\PB}{{\mathrm{PB}}}
\newcommand{\PO}{{\mathrm{PO}}}
\newcommand{\proj}{\text{-}\mathrm{proj}}
\newcommand{\inj}{\text{-}\mathrm{inj}}

\newcommand{\Eproj}{{\mathfrak E}\text{-}\mathrm{proj}}
\newcommand{\Einj}{{\mathfrak E}\text{-}\mathrm{inj}}
\newcommand{\Ideal}[1]{\operatorname{Ideal}(#1)}
\newcommand{\Ab}{\mathcal{A}b}

\begin{document}

\title[Ideal cotorsion theories in triangulated categories]{Ideal cotorsion theories in triangulated categories}
\author{Simion Breaz} \thanks{}
\address{Babe\c s-Bolyai University, Faculty of Mathematics
and Computer Science, Str. Mihail Kog\u alniceanu 1, 400084
Cluj-Napoca, Romania}
\email{bodo@math.ubbcluj.ro}

\author{George Ciprian Modoi}\thanks{}
\address{Babe\c s-Bolyai University, Faculty of Mathematics
and Computer Science, Str. Mihail Kog\u alniceanu 1, 400084
Cluj-Napoca, Romania}
\email{cmodoi@math.ubbcluj.ro}

\date{\today}


\thanks{ }

\begin{abstract}
We study ideal cotorsion pairs associated to almost exact structures in extension closed subcategories of triangulated 
categories. This approach allows us to extend the recent ideal approximations theory developed by Fu, Herzog et al. for exact categories in the above mentioned context. In the last part of the paper we apply the theory in order to study  
projective classes (in particular localization or smashing subcategories) in compactly generated categories.  
\end{abstract}


\maketitle

\tableofcontents

\section{Introduction}

Approximations of objects by some better understood ones are important tools in the study of various categories. 
For example they are used to construct resolutions and to do homological algebra: 
in module theory the existence of injective envelopes, projective precovers and flat covers are often 
used for defining derived functors, for dealing with invariants as (weak) global dimension etc. The central role in approximation theory 
for the case of module, 
or more general abelian or exact, categories is played by the notion of cotorsion pair, cf. \cite{Go-Tr}. On the other hand, in the context of triangulated categories the cotorsion pairs are replaced by $t$-structures, as in \cite{perverse}. We note that in triangulated categories   
there are no important differences between torsion and cotorsion theories. The explanation is the fact that the $\Ext^1$-functor from an abelian category may be 
computed as a shifted $\Hom$ functor in the corresponding derived category.
These structures were generalized
in \cite{Iy-Yo-08} to torsion pairs and mutation pairs, and the authors proved that some results which 
are valid for cotorsion theories in the context of module categories   
hold also in the context of triangulated categories (e.g. Wakamatsu's Lemma, \cite[Lemma 5.13]{Go-Tr}). 

On the other side, there are situations when the approximations are realized by some ideals which are not object ideals, e.g. the phantom 
ideal introduced by Ivo Herzog in \cite{Herzog-ph} (in module categories) and the ideal of $\CP$-null homomorphisms associated to a class 
$\CP$ of objects used in \cite{Ch-98} (in triangulated categories). 
In \cite{Fu-et-al} and \cite{Fu-Herzog} the authors extended, in the context of
exact categories, the study of classical cotorsion pairs to ideal cotorsion pairs, and the theory developed in \cite{Fu-et-al} 
was extended and completed in \cite{Estrada-et-al} and \cite{Ozbek}. 

Following these ideas, in the present paper we will study ideal cotorsion pairs in extension closed subcategories of triangulated categories, in order to 
obtain good information about precovering/preenveloping ideals. Since every exact category can be embedded as an extension closed subcategory
of a triangulated category (eventually extending the universe) the theory presented here includes important parts from the theory developed in \cite{Fu-et-al} and \cite{Fu-Herzog}.

Let $\CT$ be a triangulated category and $\CA$ a subcategory of $\CT$ closed under extensions. 
%
In the case of module categories precovering and preenveloping classes are studied in relation with 
the canonical exact structure on these categories. One of the main question in this context is to establish
if some or all 
precovers (preenvelopes) are deflations (inflations) with respect to this exact structure. In order to approach this problem in  
our context, let us recall that the usual substitutes for exact structures in 
triangulated categories are proper classes of triangles, introduced by Beligianis in \cite{Bel2000}. For other examples of 
applications for proper classes we refer to
\cite{Meyer2} and \cite{Meyer1}. Since we will work in subcategories of triangulated categories, and this approach covers the case of exact categories, we will use the term \textsl{almost exact structure} for such collections of triangles (i.e. 
classes of triangles in $\CA$ which are closed under  homotopy pullbacks, homotopy pushouts, finite direct sums and contains all
splitting triangles).  

Therefore, we will study precovering (preenveloping) ideals $\CI$ 
such that all $\CI$-precovers (preenvelopes) are $\frakE$-deflations (inflations), where $\frakE$ is a fixed almost exact structure in $\CA$. 
%
%
If $\CI$ is an ideal in $\CA$ then we can associate to $\CI$ the class $\mathfrak{PB}(\CI)$  
which consists in all triangles which can be constructed as homotopy pullbacks of triangles from 
$\frakE$ along maps from $\CI$. Therefore, $\mathfrak{PB}(\CI)$ is an almost exact structure contained in $\frakE$.
 Dually, if $\CJ$ is an ideal in $\CA$ then we can associate to $\CJ$ the class $\mathfrak{PO}(\CJ)$  
which consists in all triangles which can be constructed as homotopy pushouts of triangles from 
$\frakE$ along maps from $\CJ$. The class $\mathfrak{PO}(\CJ)$ is also an almost exact structure contained in $\frakE$.
A pair of ideals $(\CI,\CJ)$ is an ideal cotorsion pair with respect to  $\frakE$ if $\CI=\mathfrak{PO}(\CJ)\textrm{-proj}$ and 
$\CJ=\mathfrak{PB}(\CI)\textrm{-inj}$. Here, if $\frakF$ is an almost exact structure then 
we denote by $\frakF\textrm{-proj}$ (respectively $\frakF\textrm{-inj}$) of all maps which are
projective (injective) with respect to all triangles from $\frakF$.


We can reverse this process starting with an almost exact structure $\frakF\subseteq \frakE$. We 
associate to $\frakF$ an ideal $\mathbf{\Phi}_\frakE(\frakF)$ of those homomorphisms $\varphi$ with the property that 
all triangles obtained via  homotopy pullbacks of triangles from $\frakE$ along $\varphi$ are in $\frakF$. The elements of 
$\mathbf{\Phi}_\frakE(\frakF)$ are called \textsl{relative $\frakF$-phantoms}. Dually,
we can associate to an almost exact structure $\mathfrak{G}\subseteq \frakE$ the ideal $\mathbf{\Phi}^\frakE(\mathfrak{G})$
of all maps $\psi$ such that every triangle obtained via a homotopy pushout of a triangle from $\frakE$ along $\psi$ is a 
triangle from $\mathfrak{G}$. In this way we obtain two Galois correspondences between the class of ideals in $\CA$ and the class of 
an almost exact structures included in $\frakE$ (Theorem \ref{Galois-correspondences}).

We start with a preliminary section (Section \ref{Section-weak-proper}) where we introduce (almost) exact structures in extension closed subcategories of triangulated categories. These are natural generalizations of Quillen's exact structures and of Beligiannis proper classes. 

In Section \ref{section-precovering} we study some general properties of precovers and preenvelops associated to some (phantom) ideals. It is proved that the existence of enough injective or projective morphisms associated to an almost exact structures is connected with the existence of some ideal precovers or preenvelopes. Then we define an orthogonality associated to an almost exact structure, and we introduce special precovering and special preenveloping ideals. These are notions are natural since they describe the existence of special precovers and preenvelopes associated in module theory to complete cotorsion pairs. In this context one of the main result is Salce's Lemma which says us that in many cases all 
precovers/preenvelopes are special (i.e. they can be constructed via some special pushouts/pullbacks). 
This lemma was extended to ideals associated to exact categories in \cite{Fu-et-al}, where it is proved that an ideal $\CI$
is special precovering if and only if it is the ideal of phantoms associated to an exact structure which have enough special 
injective homomorphisms. In Theorem \ref{salce-lemma} we will prove the corresponding result for our hypothesis, and we will apply this result to obtain a characterization for complete ideal cotorsion pairs.

{%
In Section \ref{product-and-toda} we study products and extensions of ideals in order to prove that products of special precover ideals are special precover
ideals (Theorem \ref{perp-product}). This is an ideal version for Ghost Lemma, \cite[Theorem 1.1]{Ch-98}. 
 We also include here the ideal version for the above mentioned Wakamatsu's Lemma (Lemma \ref{waka}). 
It is interesting that some of the results proved for exact categories in \cite{Fu-Herzog} have simpler proofs in our context.
On the other side, working in triangulated categories we have only weak (co)kernels, and we have only homotopy pullbacks/pushouts. Therefore, we cannot use the uniqueness parts for corresponding universal properties (these are important ingredients in the case of exact categories, e.g. the reader can compare the proof for Ghost Lemma provided in \cite[Theorem 25]{Fu-Herzog} with the proof for Theorem 
\ref{perp-product}).

The main ain of Section \ref{Section-relative-cotorsion-pairs} is to provide characterizations for complete ideal torsion pairs. 
The main result is Theorem \ref{mainthA}, where it is proved that 
if we have enough $\frakE$-injective $\frakE$-inflations
and $\frakE$-projective $\frakE$-deflations, then an ideal cotorsion pair $(\CI,\CJ)$ is complete if and only if $\CI$ is a precovering ideal
or $\CJ$ is a preenveloping ideal. These complete ideal cotorsion pairs can be constructed via 
relative (co)phantoms associated to almost exact structures included in $\frakE$.


{As an application of the theory developed here we will prove in the last section of the paper 
a generalization to projective classes of a result proved by Krause in \cite{Krause-tel} for smashing subcategories of compactly 
generated triangulated categories (see Proposition \ref{krause-generalized}). More precisely, let us recall that Krause proved that every smashing subcategory $\CB$ of a compactly generated triangulated category $\CT$ is determined by the ideal  $\CI_\CB$ (in the subcategory $\CT_0$ of all compact objects in $\CT$) of 
all homomorphisms between compact objects which factorize through an element of $\CB$. We consider the same ideal $\CI_\CB$ associated to a projective 
class $(\CB,\CJ)$, and we prove that if $H:\CT\to \CA$ is a cohomological functor into a Grothendieck category $\CA$ such that $H(\CI_\CB)=0$ 
then $H$ annihilates an ideal of relative phantoms. In the case $\CB$ is smashing this ideal contains of all homomorphisms from $\CT$ which factorizes through elements from $\CB$. Using these, it is proved that a smashing subcategory satisfies the telescope conjecture if and only if it coincides to the class of objects of a phantom ideal associated to a triangulated subcategory of $\CT_0$, Proposition \ref{propo-smashing-tel}. Moreover, for the case when $H$ is full and $H(\CI_\CB)=0$ we always obtain $H(\CB)=0$ (Proposition    
\ref{H-full}).
}

For reader's convenience, the results proved in Sections \ref{Section-weak-proper}, \ref{section-precovering},  and  \ref{Section-relative-cotorsion-pairs} 
are presented together with their duals since the direct 
statements and the duals are connected in Theorem \ref{mainthA}. The direct statement is denoted by (1) and the dual is denoted by (2). 
The results proved in Section \ref{product-and-toda} can be also dualized, but we left to the reader to enunciate these duals.}

\section{Almost exact structures and ideals}\label{Section-weak-proper}

\subsection{Almost exact structures}
 We refer to \cite{Nee} for basic properties of triangulated categories. 
If $\CT$ is a triangulated category, we denote by $(-)[1]$ the suspension functor associated to $\CT$.
Moreover, $\frakD$ will denote the class of all distinguished triangles in $\CT$. Since we work only with distinguished 
triangles, by \textsl{triangle} we mean \textsl{distinguished triangle}. If 
$$\mathfrak{d}:A\overset{\alpha}\to B\overset{\beta}\to C\overset{\gamma}\to A[1]$$ is a triangle in $\CT$, we will say that $\gamma$ is the phantom map corresponding to  
$\mathfrak{d}$. For a subcategory $\CA$  of $\CT$ we denote  by $\CA^{\to}$ the class of all morphisms in $\CA$.

Let $\CT$ be a triangulated category.
If $\CA$ is a full subcategory (closed with respect to isomorphisms)
of $\CT$, we will say that it is \textsl{closed under extensions} if for every triangle $B\to C\to A\to B[1]$ in $\CT$ 
such that 
 $A$ and $B$ are objects in $\CA$ then $C$ is an object in $\CA$. We will denote by $\frakD_\CA$ the class of all 
 distinguished triangles $\mathfrak{d}:\ B\to C\to A\to B[1]$ such that $A,B,C\in\CA$ (and we will say that 
 $\mathfrak{d}$ is a \textsl{triangle from $\CA$}). Note that every extension closed subcategory $\CA$ of $\CT$ is closed 
with respect to  finite coproducts. 
We will use the following
generalization of the notion of proper class introduced in \cite{Bel2000}.

Let $\CT$ be a triangulated category, and $\CA$ a full subcategory of $\CT$ which is closed under extensions. 
A class of triangles $\frakE\subseteq \frakD_\CA$ is an \textsl{almost exact structure} for $\CA$ if 
\begin{enumerate}[{\rm (i)}]
 \item $\frakE$ is closed with respect to isomorphisms of triangles, coproducts and contains all split triangles,
 \item $\frakE$ is closed with respect to \textsl{base changes} (homotopy pullbacks) constructed along homomorphisms
 from $\CA$, i.e. if 
 $\mathfrak{d}: C\to B\to A\to C[1]$ is a triangle in $\frakE$ and $\alpha:X\to A$ is a homomorphism in $\CA$ then 
 the top triangle $\mathfrak{d}\alpha$ in every homotopy cartesian diagram 
 \[\xymatrix{           
  \mathfrak{d}\alpha:& C\ar[r]\ar@{=}[d] & Y\ar[d]\ar[r]    & X\ar[r]\ar[d]^{\alpha}   & C[1]\ar@{=}[d]\\
  \mathfrak{d}: & C\ar[r]           & B\ar[r]    & A\ar[r]     & C[1]         }\]
is in $\frakE$.
\item $\frakE$ is closed with respect to \textsl{cobase changes} (homotopy pushouts) constructed along homomorphisms
 from $\CA$, i.e. if 
 $\mathfrak{d}:\ C\to B\to A\to C[1]$ is a triangle in $\frakE$ and $\beta:C\to Z$ is a homomorphism in $\CA$ then 
 the bottom triangle $\beta\mathfrak{d}$ in every homotopy cartesian diagram 
 \[\xymatrix{           
  \mathfrak{d}: & C\ar[r]\ar[d]^\beta & B\ar[d]\ar[r]    & A\ar[r]\ar@{=}[d]   & C[1]\ar[d]^{\beta[1]}\\
  \beta\mathfrak{d}: & Z\ar[r]           & Y\ar[r]    & A\ar[r]     & Z[1]         }\]
is in $\frakE$.
\end{enumerate} 

If $\frakE$ is an almost exact structure then a triangle 
$$\mathfrak{d}: A\overset{f}\to B\overset{g} \to C\overset{\varphi}\to A[1]$$ 
which lies in  $\frakE$ will be called an \textsl{$\frakE$-triangle}.
Moreover, we will say that \begin{itemize} \item  $f$ is an 
\textsl{$\frakE$-inflation}, 
\item $g$ is an \textsl{$\frakE$-deflation}, and 
\item $\varphi$ is an \textsl{$\frakE$-phantom}. 
\end{itemize}
The class of all $\frakE$-phantoms is denoted by $\Ph(\frakE)$. 

Note that an almost exact structure $\frakE$ as in the definition above depends both on the triangulated category 
$\CT$ and the full subcategory $\CA$. Therefore, whenever we refer to an almost exact structure we assume that it is constructed in an extension 
closed subcategory $\CA$ of a triangulated category $\CT$.

We present here some standard examples of almost exact structures:

\begin{example}
Let $\CA$ be an extension closed subcategory of a triangulated category. Then $\mathfrak{D}_\CA$ is an almost exact structure in $\CA$. 
Moreover, the class $\mathfrak{D}^0_\CA$ of all splitting triangles from $\mathfrak{D}_\CA$ is also an almost exact structure. 
\end{example}

\begin{example}\label{exemplul-abelian}
Let $\CA$ be an abelian category, and we denote by $\mathbf{D}(\CA)$ the derived category associated to $\CA$. 
Then we can embed $\CA$ in $\mathbf{D}(\CA)$ as a full subcategory closed with respect to extensions by 
identifying every object $A\in\CA$ with the 
stalk complex $A^\bullet$ concentrated in $0$ such that $A^\bullet[0]=A$. Then it is obvious that the class 
$\frakE$ of all exact sequences in $\CA$ is an almost exact structure.  { Note that it is possible that the collection
of all homomorphisms $\mathbf{D}(\CA)(X,Y)$  in 
$\mathbf{D}(\CA)$ is not a set. We ignore this set theoretic difficulty since in what we do in this paper 
we can enlarge the universe. } 
\end{example}

\begin{example} 
Recall that a pair of subcategories $(\X,\mathcal{Y})$ in $\CT$ is called a \textsl{torsion theory} if $\CT(X,Y)=0$ for all $X\in\X$ and all 
$Y\in\Y$ and for every $A\in\CT$ there is a triangle \[\mathfrak{d}_A: X_A\overset{i_A}\to A\overset{j_A}\to Y_A\to X_A[1],\] 
with $X_A\in\X$ and $Y_A\in\Y$ (see \cite[Definition 2.2]{Iy-Yo-08}). Note that in this definition it is not required any closure of $\X$ and/or $\Y$ under 
shift functors. If $(\X,\Y)$ is a torsion theory in $\CT$ such that $\X[1]\subseteq\X$ and $\Y[-1]\subseteq\Y$ then $(\X,\Y[1])$ is a $t$-structure 
in the sense of \cite{perverse} (see also \cite{AJS}). 

Generalizing the previous example, let $(\X,\Y)$ be a $t$-structure in $\CT$. Then the heart of 
this $t$-structure is, 
by definition, $\CA=\X\cap\Y$. By \cite{perverse}, the full subcategory $\CA$ of $\CT$ is 
abelian. Every short 
exact sequence $0\to C\to B\to A\to 0$ in $\CA$ induces a triangle $C\to B\to A\to C[1]$ in $\CT$,
and the class of all such triangles is an almost exact structure in $\CA$. 
\end{example}

\begin{example}\label{exact-in-triang} 
If $\CA$ is an exact category in the sense of Quillen then it may be embedded as an extension closed subcategory $\CA\subseteq\CA'$ of an abelian category. A sequence of composable homomorphisms $C\to B\to A$ from $\CA$ is a conflation in $\CA$ if and only if it determines a short exact sequence in $\CA'$. Hence every conflation in $\CA$ induces a triangle 
$C\to B\to A\to C[1]$ in $\mathbf{D}(\CA')$. Since the class of all conflations is closed under 
base and cobase changes, it follows that the class of all triangles constructed as above is an almost exact structure in the subcategory $\CA$ 
of $\mathbf{D}(\CA')$.    
\end{example}

\begin{lemma}\label{infl-defl-factors}\label{lemma-factorizari-2}
Let $\frakE$ be an almost exact structure, and let $f:B\to C$ and $g:C\to Y$  be homomorphisms in $\CA$. 
\begin{enumerate}[{\rm (1)}]
\item If $f:B\to C$ is a $\mathfrak{D}_\CA$-inflation and $gf$ is an $\frakE$-inflation then  $f$ is an $\frakE$-inflation. 
\item If $g$ is a $\mathfrak{D}_\CA$-deflation and $gf$ an $\frakE$-deflation then  $g$ is an $\frakE$-deflation.
\end{enumerate}
\end{lemma}

\begin{proof}
It is enough to prove (1). Consider the triangles \[\ B\overset{f}\to C\to A\to B[1]\]  and 
\[\ B\overset{gf}\to Y\to X\to B[1],\] and we observe that they are in $\mathfrak{D}_\CA$. Then the pair $(1_B, g)$ induces
a homomorphism of triangles. The conclusion follows from the fact that $\frakE$ is closed with respect to  base changes. 
%
\end{proof}

\subsection{Ideals versus phantom ideals}

Recall that an \textsl{ideal} $\mathcal{I}$ in $\CA$ is a class of homomorphisms in $\CA^\to$ which is closed with respect to  sums of homomorphisms, 
and for every chain 
of composable  homomorphisms $A\overset{f}\to B\overset{i}\to C\overset{g}\to D$ in $\CA$, 
if $i\in\CI$ then $gif\in\CI$, or equivalently, $\CI(-,-)$ is a subfunctor of the Hom-bifunctor $\CA(-,-)$.

\begin{example}\label{exemple-obiect-ideale}
 For a class $\X$ of objects in $\CA$ which is closed with respect to
finite direct sums we put 
\[\Ideal\X=\{i\in\CA^\to\mid i\hbox{ factors through an object }X\in\CX\}.\]
It is not hard to see that $\Ideal\X$ is an ideal, and it is called the \textsl{object ideal} associated with $\X$. 

Conversely, for every ideal $\CI$ in $\CA$ we construct \textsl{the class of objects of }$\CI$ by:
\[{\rm Ob}(\CI)=\{X\in\CA\mid 1_X\in\CI\}.\]
\end{example}

Obviously, for every class $\X$ of object in $\CA$ we have ${\rm \Ob}(\Ideal\X)=\X$, and for every ideal $\CI$ of $\CA$ we have $\Ideal{{\rm Ob}(\CI)}\subseteq\CI$. An ideal is \textsl{an object ideal} 
if and only if $\Ideal{{\rm Ob}(\CI)}=\CI$. 


\begin{definition} A class of homomorphisms $\CE$ in $\CT$ is called a \textsl{phantom $\CA$-ideal} if
\begin{enumerate}[(i)] 
\item $\CE\subseteq \CT(\CA,\CA[1])=\bigcup_{A,B\in \CA}\CT(A,B[1]),$
\item $\CE$ is closed with respect to  sums of homomorphisms, 
and 
\item $\CA^{\to}[1]\CE\CA^\to\subseteq \CE$, i.e. for every chain 
of composable  homomorphisms $$A\overset{f}\to B\overset{i}\to C[1]\overset{g[1]}\to D[1]$$ in $\CT$ 
such that $i\in\CE$ and $f,g\in\CA^\to$ we have $g[1]if\in\CE$.
\end{enumerate}
\end{definition}
 
\begin{remark}\label{rem-phantom-ideal} \begin{enumerate}[(a)] \item Since we work in additive categories, as in the case of ideals, we can replace the condition (ii) in the definition of 
phantom $\CA$-ideals by 
\begin{itemize}
\item[(ii')] $\CE$ is closed with respect to finite direct sums of homomorphisms.  
\end{itemize}
\item A phantom $\CA$-ideal $\CE$ can be defined equivalently as a subfunctor $\CE:\CA^{op}\times\CA[1]\to\Ab$, 
$\CE(A,C)=\{\varphi:A\to C[1]\mid \varphi\in \CE\}$ of the shifted 
Hom bifunctor $\CT(-,-[1])$.
 \item If $\CA[1]=\CA$, in particular for $\CA=\CT$, then  a class of homomorphisms $\CI$ is a phantom $\CA$-ideal if and only if it is an ideal in $\CA$.
\end{enumerate}
\end{remark}





In fact, as in the case of proper classes studied in \cite{Bel2000}, there is an 1-to-1 correspondence between almost exact structures and 
phantom $\CA$-ideals. This is described in the following proposition, whose proof is a simple exercise.

\begin{proposition}\label{pseudo-ph-vs-wpc}
Let $\CT$ be a triangulated category, and $\CA$ a full subcategory which is closed under extensions. 
The following are equivalent for class $\frakE\subseteq \frakD_\CA$ of triangles in $\CA$ which is closed 
with respect to isomorphisms:
\begin{enumerate}[{\rm (a)}]
 \item $\frakE$ is an almost exact structure;
 \item $\CA^{\to}[1]\Ph(\frakE)\CA^\to\subseteq \Ph(\frakE)$ and $\Ph(\frakE)$ is closed with respect to (direct) sums of homomorphisms.
\end{enumerate}

Consequently, \begin{enumerate}[{\rm (i)}]
               \item              
If $\frakE$ is an almost exact structure from $\CA$ then $\Ph(\frakE)$ is a phantom $\CA$-ideal, 
and 
\item for every phantom $\CA$-ideal $\CI$ the class $$\frakD(\CI)=\{\mathfrak{d}\in\frakD\mid \hbox{the phantom of }\mathfrak{d}\hbox{ is in }\CI\}$$ 
is an almost exact structure.
\item The correspondences from {\rm (i)} and {\rm (ii)} above are inverse to each other. 
\end{enumerate}
\end{proposition}

\begin{example}
Let $\mathcal{B}$ be a class of objects in $\CT$ { which is closed with respect to finite direct sums,} 
and let $\frakE$ be an almost exact structure in $\CA$. 
Then $\CB$ induces 
an almost exact structure $\frakF_\CB\subseteq \frakE$ defined by the condition 
$$\Ph(\frakF_\CB)=\{\varphi\in\Ph(\frakE)\mid \varphi \text{ factorizes through an object } X\in\CB\}.$$ 

As a particular example we mention that the ideal $\CI$ used in \cite[Theorem A]{Krause-tel} can be viewed as a phantom 
$\CA$-ideal: If $\CT$ is a compactly generated triangulated category and $\CT_{0}$ is the full subcategory of all compact
objects in $\CT$ then every class $\mathcal{B}$ of objects in $\CT$ which is closed with respect to direct sums induces an ideal 
$\CI_{\mathcal{B}}$ in $\CT_0$ which consists in all homomorphisms between compact objects which factorize through objects from $\CB$.
\end{example}

We already noticed that phantom $\CA$-ideals and ideals in $\CA$ are different notions, unless $\CA$ is closed under suspension. For avoiding confusions, and having in mind the above correspondence,  
we prefer to work with almost exact structures instead of phantom $\CA$-ideals, whenever this is possible. However we 
keep the notion of phantom ideals because the particular situation when they are genuine ideals is a motivating example (see \cite{Bel2000}).

Moreover, as in \cite[Section 2.4]{Bel2000}, we can apply Baer's theory techniques to almost exact structures. 
Let $\frakE$ be an almost exact structure.  
Two $\frakE$-triangles 
$\mathfrak{d}$ and $\mathfrak{d'}$ as in the next commutative diagram  are \textsl{equivalent} if there is a homomorphism of triangles of the form:
\[ \xymatrix{\mathfrak{d}: & C\ar[r]\ar@{=}[d] & B'\ar[d]^{\beta}\ar[r]& A\ar@{=}[d]\ar[r] & C[1]\ar@{=}[d]\\
	     \mathfrak{d'}:& C\ar[r]      & B\ar[r]     & A\ar[r]           & C[1].}\]
In this case we know that $\beta$ has to be an isomorphism, and we defined an equivalence relation on the class of all $\frakE$-triangles 
starting in $C$ and ending in $A$. Since two $\frakE$-triangles are equivalent if and only if they have the same phantom, it is easy to see that the class of all $\frakE$-triangles starting in $C$ and ending in $A$ is a set modulo the equivalence of triangles. We denote by $\frakE(A,C)$ this set.
Using base and cobase changes we can define a sum on the set $\frakE(A,C)$, and we have an additive bifunctor:
\[\frakE(-,-):\CA^{op}\times\CA\to\Ab\]
which associates to every pair $(A,C)$ of objects from $\CA$ the group $\frakE(A,C)$ of all $\frakE$-triangles 
$C\to B\to A\to C[1]$ modulo the equivalence of triangles. It is not hard to see that by assigning to each triangle its phantom map we get an isomorphism of bifunctors
\[\frakE(-,-)\to\Ph(\frakE)(-,-[1]).\]

\begin{remark} (\textsl{Base-cobase and cobase-base changes})
Let $\frakE$ be an almost exact structure and let $\mathfrak{d}:\  C\to B\to A\overset{\varphi}\to C[1]$ be a triangle in $\frakE$. 
If $\alpha:X\to A$ and $\beta:C\to Y$ are two homomorphisms, we can construct a triangle $\mathfrak{d}\alpha$ as a 
homotopy pullback of $\mathfrak{d}$ along $\alpha$, and then a triangle $\beta(\mathfrak{d}\alpha)$ as a homotopy pushout of
$\mathfrak{d}\alpha$ along $\beta$. We can also construct a triangle $\beta \mathfrak{d}$ as a homotopy pushout, and then 
$(\beta\mathfrak{d})\alpha$ as a homotopy pullback. It is easy to see that both triangles $\beta(\mathfrak{d}\alpha)$ and $(\beta\mathfrak{d})\alpha$ have the same phantom map, namely $\beta[1]\varphi\alpha$, hence they  are equivalent.
\end{remark}



An almost exact structure $\frakE$ is called an \textsl{exact structure} provided that it satisfies one of the equivalent conditions in the following:  

\begin{lemma} Let $\CA$ be a full subcategory  of a triangulated category $\CT$. If $\CA$ is closed under extensions   and $\frakE\subseteq\frakD_\CA$ is an almost exact structure, the following are equivalent: 
\begin{enumerate}[{\rm (a)}]
\item If $A,C,Y\in\CA$, $i:C\to Y$ is an $\frakE$-inflation and $\phi:A\to C[1]$, then $i[1]\phi\in\Ph(\frakE)$ implies $\phi\in\Ph(\frakE)$. 
\item  If the commutative diagram  
\[\xymatrix{           
  \mathfrak{d}: & C\ar[r]\ar[d]^{i} & B\ar[d]\ar[r]    & A\ar[r]\ar@{=}[d]   & C[1]\ar[d]^{i[1]}\\
i\mathfrak{d}:  & Y\ar[r]           & Z\ar[r]    & A\ar[r] & Y[1]          
}\] is obtained from the triangle $\mathfrak{d}\in\mathfrak{D}_\CA$ by a cobase change along an $\frakE$-inflation $i$,
such that the bottom triangle is in $\frakE$, then the top triangle $\mathfrak{d}$ lies also in $\frakE$.
\item If $A,C,X\in\CA$, $p:X\to A$ is an $\frakE$-deflation and $\phi:A\to C[1]$ then $\phi p\in\Ph(\frakE)$ implies $\phi\in\Ph(\frakE)$. 
\item If the commutative diagram  
\[\xymatrix{           
  \mathfrak{d}p: & C\ar[r]\ar@{=}[d] & Y\ar[d]\ar[r]    & X\ar[r]\ar[d]^{p}   & C[1]\ar@{=}[d]\\
  \mathfrak{d}: & C\ar[r]           & B\ar[r]    & A\ar[r] & C[1]          \\
}\] is obtained from the bottom triangle $\mathfrak{d}\in\mathfrak{D}_\CA$ by base change along an $\frakE$-deflation $p$,
such that the top triangle $\mathfrak{d}p$ is in $\frakE$, then $\mathfrak{d}\in \frakE$.
\end{enumerate}
\end{lemma}

\begin{proof}
The equivalences (a)$\Leftrightarrow$(b) and (c)$\Leftrightarrow$(d) are obvious. Moreover, (a)$\Rightarrow$(c) and (c)$\Rightarrow$(a) are dual to each other, so it
is enough to prove (a)$\Rightarrow$(c).

{ Let $p:X\to A$ be an $\frakE$-deflation and let $\phi:A\to C[1]$ be a map such that $A,C,X\in\CA$ and $\phi p\in\Ph(\frakE)$. Completing both $p$ and $\phi p$ to triangles we obtain the following commutative diagram: 
\[\xymatrix{           
 & B\ar[r]\ar[d] & X\ar@{=}[d]\ar[r]^{p}    & A\ar[r]^{\psi}\ar[d]^{\phi}   & B[1]\ar[d]\\
 C\ar[r]^{i} & Y\ar[r]           & X\ar[r]^{\phi p}    & C[1]\ar[r]^{-i[1]} & Y[1]          \\
}\]  
By hypothesis, $\psi\in\Ph(\frakE)$, $i$ is an $\frakE$-inflation and $i[1]\phi\in\Ph(\frakE)$. Then (a) implies
$\phi\in\Ph(\frakE)$. 
}
\end{proof}


\begin{example}\label{fantome-clasice-gen} For the case $\CA=\CT$, exact structures which are closed under all suspensions are studied in \cite{Bel2000} under the name \textsl{proper classes of triangles}. Recall that in this case phantom ideals coincide with genuine ideals. 
We mention here a basic example: If $\mathcal{H}$ is a class of objects in $\CT$ such that 
$\mathcal{H}[1]=\mathcal{H}$ then the class $\frakE_\mathcal{H}$ of all triangles $A\to B\to C\to A[1]$ such that 
the sequences of abelian groups $$0\to \CT(H,A)\to \CT(H,B)\to \CT(H,C)\to 0$$ are exact for all $H\in \mathcal{H}$ is 
an exact structure. It is easy to see that 
$$\Ph(\frakE_\mathcal{H})=\{f\mid \CT(H,f)=0 \text{ for all } H\in\mathcal{H}\}.$$ Since $\mathcal{H}$ is closed under suspensions (i.e. 
$\mathcal{H}[1]=\mathcal{H}$) then $\Ph(\frakE_\mathcal{H})$ is also closed under suspensions.  

{ In particular, we mention here the case when $\CT$ is compactly generated and $\mathcal{H}$ is the class of all 
compact objects in $\CT$. Then $\Ph(\frakE_\mathcal{H})$ is the class of classical phantom maps (the maps $\phi$ 
for which $\CT(H,\phi)=0$ for all compacts $H\in\CT$, \cite{Krause-tel}). Actually this example motivates the name `phantom' chosen for
$\Ph(\frakE)$.  }
\end{example}




The next proposition shows that relative to an exact structure the composition of two inflations (deflations) 
must be an inflation (respectively a deflation). Therefore exact structures satisfy all triangulated versions for the axioms 
of  exact categories
(see for example \cite[Definition 2.1]{Bu10}).  

\begin{proposition}\label{composed-inflations}
 If $\frakE$ is a exact structure then 
 \begin{enumerate}[{\rm (1)}]
  \item The composition of two $\frakE$-inflations is an $\frakE$-inflation.
  \item The composition of two $\frakE$-deflations is an $\frakE$-deflation.
 \end{enumerate}
\end{proposition}

\begin{proof}
Let $f:A\to B$ and $g:B\to C$ be two $\frakE$-inflations. Completing them to triangles and using 
the octehedral axiom we construct the commutative diagram, whose rows and columns are triangles:
\[\xymatrix{
                      & C[-1]\ar@{=}[r]\ar[d]                 & C[-1]\ar[d]                   & \\
X[-1]\ar[r]\ar@{=}[d] & Y[-1]\ar[r]^{h[-1]}\ar[d]_{-\psi[-1]} & Z[-1]\ar[r]\ar[d]^{-\phi[-1]} & X\ar@{=}[d]\\
X[-1]\ar[r]           & A\ar[r]^{f}\ar[d]_{gf}               & B\ar[r]\ar[d]^{g}             & X\\
                      & C\ar@{=}[r]                           &C                              &
}.\]
We have $f\psi[-1]=\phi[-1]h[-1]\in\Ph(\frakE)[-1]$ and saturation gives us $\psi\in\Ph(\frakE)$ 
since $f$ is an $\frakE$-inflation. Therefore $gf$ is an $\frakE$-inflation too.
\end{proof}

\section{Precovering and preenveloping ideals}\label{section-precovering}

\subsection{Precovers and preenvelopes}
Let $\CI$ be an ideal in $\CA$, and $A$ an object in $\CA$. 
We say that a homomorphism $i:X\to A$ is an \textsl{$\CI$-precover} for $A$ 
if $i\in\CI$ and all homomorphisms $i':X'\to A$ from $\CI$ factorize through $i$. Dually, 
an \textsl{$\CI$-preenvelope} for an object $B$ in $\CA$ is a homomorphism 
$i:B\to Y$ which lies in $\CI$ such that every other homomorphism $i':B\to Y'$ from $\CI$ 
factorizes through $i$. The ideal $\CI$ is a \textsl{precovering} (\textsl{preenveloping}) ideal
if every object from $\CA$ has an $\CI$-precover ($\CI$-preenvelope).

Because the suspension functor is an equivalence we deduce immediately that, for every $n\in\Z$, 
$i:X\to A$ is an $\CI$-precover for $A$ if and only if $i[n]:X[n]\to A[n]$ is an $\CI[n]$-precover for 
$A[n]$, and a similar statement holds for preenvelopes too. 

We extend these notions for phantom $\CA$-ideals in the following way: if $\CE$ is a phantom $\CA$-ideal and $A$ is an object in $\CA$, 
we say that a homomorphism $\phi:X\to A[1]$ is an \textsl{$\CE$-precover} for $A[1]$ 
if $\phi\in\CE$ and all homomorphisms $\phi':X'\to A[1]$ in $\CE$ factorize through $\phi$. Dually, 
an \textsl{$\CE$-preenvelope} for an object $B$ in $\CA$ is a homomorphism 
$\phi:B\to Y[1]$ which lies in $\CE$ such that every other homomorphism $\phi':B\to Y'[1]$ from $\CE$ 
factorizes through $\phi$. The phantom 
$\CA$-ideal $\CE$ is \textsl{precovering} (\textsl{preenveloping})
if every object from $\CA[1]$ (resp. $\CA$) has an $\CE$-precover ($\CE$-preenvelope). 

In the following we will see that precovers (preenvelopes) are strongly connected with some injective (respectively, projective) 
properties. 

Let $\frakE$ be an almost exact structure in $\CA$. 
We say that a homomorphism $f:X\to A$ from $\CA$ is 
\textsl{$\frakE$-projective} if $f$ is projective with respect to  all triangles in $\frakE$, 
i.e. for every triangle $C\overset{\alpha}\to B\overset{\beta}\to A\overset{\phi}\to C[1]$ in $\frakE$ there is a homomorphism
$\overline{f}:X\to B$ such that $f=\beta\overline{f}$ ($f$ factorizes through all $\frakE$-deflations $B\to A$). 
Dually, $g:C\to Y$ is \textsl{$\frakE$-injective} if $g$ is injective with respect to  all triangles in $\frakE$, i.e.
$f$ factorizes through all $\frakE$-inflations $C\to B$. 
We denote by $\Eproj$ ($\Einj$) the class of all $\frakE$-projective 
(respectively, $\frakE$-injective) homomorphisms.

The proof of the following lemma is straightforward.

\begin{lemma}\label{basic-CI-inj}
Let $\frakE$ be an almost exact structure in $\CA$. 
\begin{enumerate}[{\rm (a)}]
 \item A homomorphism $f:X\to A$ from $\CA$ is $\frakE$-projective if and only if $\Ph(\frakE) f=0$.
 \item  A homomorphism $g:C\to Y$ from $\CA$ is $\frakE$-injective if and only if $g[1]\Ph(\frakE)=0$.
%
 \item $\Eproj$ and $\Einj$ are ideals in $\CA$.
 \item $\frakE[n]\text{-}\mathrm{proj}=(\Eproj)[n]$ and $\frakE[n]\text{-}\mathrm{inj}=(\Einj)[n]$ for all $n\in\Z$
 (where $\frakE[n]$ is viewed as an almost exact structure relative to the full subcategory $\CA[n]$).  
\end{enumerate}
\end{lemma}

{
\begin{corollary}\label{direct-sum-inj-proj}
Let $\frakE$ be an almost exact structure in 
 an extension closed subcategory $\CA$ of the triangulated category $\CT$.
\begin{enumerate}[{\rm (1)}]
\item A homomorphism $\alpha:A\to \prod_{i\in I}B_i$ is $\frakE$-injective if and only if for every $i\in I$ the homomorphism 
$\pi_i\alpha$ is $\frakE$-injective ($\pi_i:\prod_{i\in I}B_i\to B_i$ denote the cannonical projections).

\item A homomorphism $\alpha:\oplus_{i\in I}B_i\to A$ is $\frakE$-projective if and only if for every $i\in I$ the homomorphism 
$\alpha\rho_i$ is $\frakE$-projective ($\rho_i:B_i\to \oplus_{i\in I}B_i$ denote the cannonical injections).
\end{enumerate} 
\end{corollary}

\begin{proof}
This follows from the fact that the family of all canonical projections (injections) associated to a direct product (direct sum) is 
monomorphic (epimorphic).
\end{proof}
}


The above mentioned connection is presented in the following results:

\begin{lemma}\label{Icov-vs-Iinjenv} Let $\frakE$ be an almost exact structure in 
 an extension closed subcategory $\CA$ of $\CT$, and let 
  $$C\overset{f}\to B\overset{g}\to A\overset{\phi}\to C[1]$$ be an $\frakE$-triangle.
  \begin{enumerate}[{\rm (1)}]  
\item If  $\phi$ is an $\Ph(\frakE)$-precover for $C[1]$ and $d:X\to B$ is a homomorphism
 such that $gd=0$ then $d\in\Einj$. In particular $f$ is an $\Einj$-preenvelope of $C$. Consequently,
the map $\phi$ is an $\Ph(\frakE)$-precover for $C[1]$ if and only if $f$ is $\frakE$-injective.
 
\item If  $\phi$ is an $\Ph(\frakE)$-preenvelope for $A$ and $d:B\to X$ is a homomorphism such that 
$df=0$ then $d\in\Eproj$. In particular $g$ is an $\Eproj$-precover for $A$. Consequently, 
the map $\phi$ is an $\Ph(\frakE)$-preenvelope for $A$ if and only if $g$ is $\frakE$-projective.
\end{enumerate}
\end{lemma}

\begin{proof} It is enough to prove (1).

 Let $\psi:Y\to X[1]$ be a homomorphism in $\Ph(\frakE)$. 
Our initial data consist in the solid part of the following (commutative) diagram:

\[ \xymatrix{ & X\ar@{-->}[dl]_{h}\ar[d]^{d} &  & Y \ar@{-->}[dl]_{k}\ar[r]^{\psi} & X[1]\ar@{-->}[dl]_{h[1]}\ar[d]^{d[1]}   \\
 C \ar[r]_{f}& B\ar[r]_{g}  & A \ar[r]_{\phi} & C[1]\ar[r]_{f[1]} & B[1] 
 .} \]


Since $gd=0$ and $f$ is a weak kernel 
 for $g$ we get a factorization $d=fh$ for some homomorphism $h:X\to C$. Now $h[1]\psi:Y\to C[1]$ is in $\Ph(\frakE)$ 
 because $\Ph(\frakE)$ is an phantom $\CA$-ideal. Since $\phi$ 
 is an $\Ph(\frakE)$-precover for $C[1]$, 
 we get further a factorization $h[1]\psi=\phi k$, for some $k:Y\to A$. 
 Now $d[1]\psi=f[1]h[1]\psi=f[1]\phi k=0$. Therefore $d[1]\Ph(\frakE)=0$ so $d\in\Einj$. 
 
 Since $gf=0$ we have $f\in\Einj$.
 Moreover, if 
 $f':C\to B'$ is a homomorphism in $\Einj$, then $f'\phi[-1]=0$. Since $f$ is a weak cokernel for $\phi[-1]$, $f'$ has to factor through $f$.  

For the last statement, let us observe that if $\phi$ is a precover for $C[1]$ then $f$ is $\frakE$-injective by 
what we just proved. 
Conversely, if $f\in\Einj$ then for every map $\psi:X\to C[1]$ from $\Ph(\frakE)$ we 
have $f[1]\psi=0$. But $\phi$ is a weak kernel for $f[1]$, hence $\psi$ factors through $\phi$.
%
 \end{proof}

\begin{definition}\label{def-suficiente}
Let $\frakE$ be an almost exact structure in $\CA$. We say that there are \textsl{enough $\frakE$-injective homomorphisms}  
if for every object $A$ there 
 exists an $\frakE$-inflation $f:A\to B$ which is $\frakE$-injective.

Dually, there are \textsl{enough $\frakE$-projective homomorphisms} if for every object $C$ there 
 exists an $\frakE$-deflation $g:B\to C$ which is $\frakE$-projective;
\end{definition}

\begin{theorem}\label{Th-procov-vs-Iinj}
Let $\frakE$ be an almost exact structure in $\CA$.
	\begin{enumerate}[{\rm (1)}]
\item The following are equivalent: \begin{enumerate}[{\rm (a)}]
\item  there are enough $\frakE$-injective homomorphisms;
 \item $\Ph(\frakE)$ is a precovering phantom $\CA$-ideal.  
 
 \end{enumerate}
\item The following are equivalent:  \begin{enumerate}[{\rm (a)}] \item there are enough $\frakE$-projective homomorphisms;
 \item $\Ph(\frakE)$ is a preenveloping phantom $\CA$-ideal.
 \end{enumerate}
\end{enumerate}
\end{theorem}

\begin{proof} 

(a)$\Rightarrow$(b) Let $A$ be an object in $\CA$. Using the hypothesis we observe that there exists an $\frakE$-triangle
$A\overset{f}\to B\to C\overset{\phi}\to A[1]$ such that $f$ is $\frakE$-injective. By (a) and Lemma \ref{Icov-vs-Iinjenv}  
we conclude that $\phi$ is a 
$\Ph(\frakE)$-precover for $A[1]$.

(b)$\Rightarrow$(a)
Suppose that $\Ph(\frakE)$ is a precovering phantom $\CA$-ideal. Let $A$ be an object in $\CA$. 
If $\phi:C\to A[1]$ is an $\Ph(\frakE)$-precover, we consider the $\frakE$-triangle 
$A\overset{f}\to B\to C\overset{\phi}\to A[1]$. Using Lemma \ref{Icov-vs-Iinjenv}, we
conclude that $f$ is an $\frakE$-injective $\frakE$-inflation.
\end{proof}

In the following we will present a method to construct almost exact structures with enough injective/projective homomorphisms. 
This is an extension of the method presented in \cite[Section 1]{Aus-So}. We start with an example of an almost exact structure which extends
Example \ref{fantome-clasice-gen}.

\begin{example}
Let $\CI$ be an ideal in $\CA$. Then $$\mathcal{E}^{\CI}=\{\varphi\in\CT(\CA,\CA[1])\mid \CI[1]\varphi=0\}$$ is a phantom $\CA$-ideal,
hence the class $\frakE^\CI=\frakD(\mathcal{E}^{\CI})$ is an almost exact structure in $\CA$ according to Proposition \ref{pseudo-ph-vs-wpc}.
It is easy to see that $\frakE^\CI$ is the class of all triangles 
$C\to B\to A\to C[1]$ from $\mathfrak{D}_\CA$ with the property that all homomorphisms from $\CI$ are injective with respect to 
these triangles. 

Dually, if we consider the phantom $\CA$-ideal $$\mathcal{E}_{\CI}=\{\varphi\in\CT(\CA,\CA[1])\mid \varphi\CI=0\},$$
we obtain the almost exact structure $\frakE_\CI$ of all triangles 
$C\to B\to A\to C[1]$ from $\mathfrak{D}_\CA$ with the property that all homomorphisms from $\CI$ are projective with respect to 
these triangles. 
\end{example}

\begin{proposition}\label{prop1-suficiente}
Let $\frakE$ be an almost exact structure in $\CA$. 
\begin{enumerate}[{\rm (1)}]
\item If there are enough $\frakE$-injective homomorphisms then $\frakE=\frakE^{\frakE\textrm{-}\mathrm{inj}}$.
\item If there are enough $\frakE$-projective homomorphisms then $\frakE=\frakE_{\frakE\textrm{-}\mathrm{proj}}$.
\end{enumerate} 
\end{proposition}

\begin{proof}
(1) Let $\CI$ be the ideal $\frakE\textrm{-inj}$. It is enough to prove the inclusion $\frakE^\CI\subseteq \frakE$.

Let $\mathfrak{d}: C\to B\to A\to C[1]$ be a triangle in $\frakE^\CI$. If $\alpha:C\to E$ is an $\frakE$-injective $\frakE$-inflation then
$\alpha\in \CI$, so it is injective with respect to $\mathfrak{d}$. Therefore, we can construct a commutative diagram  
\[ \xymatrix{\mathfrak{d}: & C\ar[r]\ar@{=}[d] & B\ar[d]\ar[r]& A\ar[d]\ar[r] & C[1]\ar@{=}[d]\\
	     & C\ar[r]      & E\ar[r]     & D\ar[r]           & C[1],}\]
where the horizontal lines are triangles in $\mathfrak{D}_\CA$. Since the below triangle is in $\frakE$, it follows that the top triangle is also in $\frakE$, and the proof is complete.

(2) This is the dual of (1).
\end{proof}

It is easy to see that if $\frakF$ and $\frakE$ are almost exact structures such that $\frakF\subseteq \frakE$ then $\frakE\inj\subseteq 
\frakF\inj$ and $\frakE\proj\subseteq \frakF\proj$. We can use the previous proposition to prove a converse for this implication. 

\begin{corollary}\label{cor-prop1-suficiente}
Let $\frakE$ and $\frakF$ be almost exact structures in $\CA$. 
\begin{enumerate}[{\rm (1)}]
\item Suppose that there are enough $\frakE$-injective homomorphisms. Then $\frakF\subseteq \frakE$ 
if and only if $\frakE\inj\subseteq \frakF\inj$.
\item Suppose that there are enough $\frakE$-projective homomorphisms. Then $\frakF\subseteq \frakE$ 
if and only if $\frakE\proj\subseteq \frakF\proj$.
\end{enumerate} 
\end{corollary}

\begin{proof}
(1) Suppose that $\CI=\frakE\inj\subseteq \frakF\inj=\CJ$. Since we have enough $\frakE$-injective homomorphisms, we can apply the 
previous proposition to obtain $\frakF\subseteq \frakE^\CJ\subseteq \frakE^\CI=\frakE$.  
\end{proof}

\begin{proposition}\label{subclasses-enough-inj}
Let $\frakF\subseteq \frakE$ be almost exact structures in $\CA$.  
\begin{enumerate}[{\rm (1)}]
\item If there are enough $\frakE$-injective homomorphisms the following are equivalent:  \begin{enumerate}[{\rm (a)}]
\item there are enough $\frakF$-injective homomorphisms;
\item  there exists a preenveloping ideal $\CI$ in $\CA$ such that $\frakE\textrm{-}\mathrm{inj}\subseteq \CI$, and $\frakF=\frakE^\CI$.
\end{enumerate}

\noindent In these conditions $\frakF\inj=\CI$.
\item If there are enough $\frakE$-projective homomorphisms the following are equivalent: \begin{enumerate}[{\rm (a)}] 
\item there are enough $\frakF$-projective homomorphisms;
\item  there exists a precovering ideal $\CI$ in $\CA$ such that $\frakE\textrm{-}\mathrm{proj}\subseteq \CI$ and $\frakF=\frakE_\CI$.
 \end{enumerate}

\noindent In these conditions $\frakF\proj=\CI$.
\end{enumerate}
\end{proposition}

\begin{proof} We will prove (1). 

(a)$\Rightarrow$(b) Take $\CI=\frakF\textrm{-}\mathrm{inj}$. The conclusion follows from Definition \ref{def-suficiente} and Proposition
\ref{prop1-suficiente}.

(b)$\Rightarrow$(a) We have to prove that $\frakE^\CI$ has enough injective homomorphisms. Let $A$ be in $\CA$. 
We start with an $\frakE$-triangle $A\overset{e}\to E \to X\to A[1]$ such that $e$ is $\frakE$-injective. Let $i:A\to I$ be an 
$\CI$-preenvelope for $A$. Since $e\in \CI$, it factorize through $i$. Then we have a commutative diagram 
\[ \xymatrix{ A\ar[r]^{i}\ar@{=}[d] & I\ar[d]\ar[r]& Z\ar[d]\ar[r] & A[1]\ar@{=}[d]\\
	      A\ar[r]^{e}      & E\ar[r]     & X\ar[r]           & A[1],}\]
and using the closure of $\frakE$ with respect to  base changes we conclude that $i$ is an $\frakE$-inflation.
Since $i$ is an $\CI$-preenvelope, it is easy to see that the triangle $A\overset{i}\to I\to Z\to A[1]$ is in $\frakE^\CI$, hence
$i$ is an $\frakE^\CI$-injective $\frakE^\CI$-inflation.

For the last statement, let us observe that for every $A\in \CA$ every $\CI$-preenvelope $i:A\to X$ is a $\frakF$-inflation. 
Therefore, every $\frakF$-injective homomorphism $A\to X$ factorizes through $i$,
hence $\frakF\inj\subseteq \CI$. The converse inclusion follows from the equality 
$\frakF=\frakE^\CI$.           
\end{proof}

{
For further reference we mention here the following particular case:

\begin{example}\label{ex-constr-suf-proj}
If $\CA=\CT$ and $\frakE=\frakD$ is the class of all triangles in $\CT$ then $0$ is the ideal of all $\frakD$-injective ($\frakD$-projective)
homomorphisms. Since in this case all homomorphisms are $\frakD$-inflations and $\frakD$-deflations, it follows that we have enough 
$\frakD$-injective homomorphisms and $\frakD$-projective homomorphisms. 

If $\CI$ is a preenveloping (precovering) ideal, we consider the almost exact structure $\frakD^\CI$ (resp. $\frakD_\CI$) of all triangles $\mathfrak{d}$ such that all $i\in \CI$ are injective (projective) relative to $\mathfrak{d}$. By what we just proved we obtain that $\frakD^\CI$ 
(resp. $\frakD_\CI$) has enough injective (projective) homomorphisms and $\frakD^\CI\inj=\CI$ (resp. $\frakD_\CI\proj=\CI$).  
\end{example}
}


%




\subsection{Orthogonality}

We say that a homomorphism $f:X\to A$ from $\CA$ is \textsl{left orthogonal} 
(with respect to $\frakE$) 
to a homomorphism $g:B\to Y$ from $\CA$, 
and we denote this by $f\perp g$, 
if $$\CT(f,g[1])(\Ph(\frakE))=0,$$ i.e. for all homomorphisms $\phi:A\to B[1]$ in $\Ph(\frakE)$ 
we have $g[1]\phi f=0$.  
{
This means that for every triangle $B\to C\to A\overset{\phi}\to B[1]$ from $\frakE$ the triangle obtained by a base-cobase change
 \[ \xymatrix{ B\ar[r]\ar@{=}[d] & C\ar[r]\ar@{<-}[d] & A \ar[r]^{\phi}\ar@{<-}[d]^{f} & B[1]\ar@{=}[d]\\
B\ar[r]\ar[d]^{g} & C'\ar[r]\ar[d] & X \ar[r]\ar@{=}[d] & B[1]\ar[d]^{g[1]} \\
Y\ar[r] & C''\ar[r] & X \ar[r]^{0} & Y[1]
}\] splits.
}


\begin{example}
(a) If $\CA$ is an abelian category, $\CT=\mathbf{D}(\CA)$, and $\frakE$ is the class of all short exact 
sequences in $\CA$ then $f\perp g$ if and only if $\Ext(f,g)=0$.

(b) If $\CA$ is an exact category, $\CT=\mathbf{D}(\CA')$, where $\CA'$ is an abelian category containing $\CA$ as an 
extension closed category, and $\frakE$ is the class of triangles coming from conflations in $\CA$, then 
$f\perp g$ if and only if $\Ext(f,g)=0$. 

(c) If $\CA=\CT$, and $\frakE$ is an almost exact structure in $\CT$ then $f\perp g$ means exactly $\Ph(\frakE)(f,g)=0$ (we use Remark \ref{rem-phantom-ideal}(b)). 
In particular, if
$\frakE=\frakD$ is the class of all triangles in $\CT$ then $f\perp g$ iff $\CT(f,g[1])=0$.
\end{example}

\begin{remark}
Let $\CA$ be an abelian  category, 
$\CT=\mathbf{D}(\CA)$ 
and let $\frakE=\frakD$ be the class of all triangles in $\CT$. 
Then a homomorphism $f:A_0\to A_1$ in $\CA$ may be interpreted as a 
complex, so it gives rise to an object $\mathbf{f}$ of $\mathbf{D}(\CA)$. Clearly if $g:B_0\to B_1$ is another homomorphism
in $\CA$, and $\mathbf{g}$ is the object in $\mathbf{D}(\CA)$ by the complex induced by $f$ then we may consider 
the condition ${\mathbf{D}(\CA)}(\mathbf{f},\mathbf{g}[1])=0$, that is there is no other map in $\mathbf{D}(\CA)$ between the 
two complexes above than $0$ (for example such a condition appears in the definition of a (pre)silting object 
in \cite{AMV}). We want to warn that  this kind of orthogonality is different from ours. 

Indeed, for $\CA=\Ab$ the 
category of all abelian groups, let us consider the homomorphism $f:\Z\to \Z$ which is the multiplication by $2$. Since $f$ has projective domain 
it follows easily that $\phi f=0$ for all $\phi:\Z\to X[1]$, with $X\in\Ab$. In fact since $\phi$ has also projective domain, 
even $\phi$ vanishes. Therefore $\mathbf{D}(\Ab)(f,f[1])=0$, hence $f\perp f$. 
On the other side, as object in $\mathbf{D}(\Ab)$, $\mathbf{f}$ is a bounded complex of projectives, so it is homotopically projective. 
It follows that the homomorphisms in $\mathbf{D}(\Ab)$ starting in $\mathbf{f}$ are exactly homotopy classes of 
homomorphisms of complexes. But for every two homomorphisms of abelian groups $s_0,s_1$ as in the diagram:
\[\xymatrix{\mathbf{f}: & \cdots\ar[r] & 0\ar[r] & \Z\ar[r]^{f}\ar[dl]_{s_0}&\Z\ar[r]\ar[dl]^{s_1}  & \cdots \\
\mathbf{f}[1]: & \cdots\ar[r] &\Z\ar[r]^{f} &\Z\ar[r] &0\ar[r]  & }\]
we have $\im(fs_0+s_1f)\subseteq2\Z$ showing that the homomorphism of complexes which is the identity map in
degree 0 and 0 otherwise is not homotopic to 0, and it follows that $\mathbf{D}(\Ab)(\mathbf{f},\mathbf{f}[1])\neq 0$. 
\end{remark}

\begin{lemma}\label{lema-ort-basic}
Let $f:X\to A$ and $g:B\to Y$ be two homomorphisms in $\CA$. The following are equivalent:
\begin{enumerate}[{\rm (1)}]
 \item $f\perp g$;
 \item every homomorphism $\phi:A\to B[1]$ from $\Ph(\frakE) $ induces a triangle homomorphism 
 \[ \xymatrix{ Z \ar[r] \ar@{-->}[d] & X\ar[r]^f \ar@{-->}[d] & A \ar[r] \ar[d]^{\phi} & Z[1]\ar@{-->}[d]  \\
 Y \ar[r]& T\ar[r]  & B[1] \ar[r]_{g[1]}  & Y[1] 
 .}
 \]
 \end{enumerate}
\end{lemma}

\begin{proof}
Let $\phi:A\to B[1]$ be a homomorphism in $\Ph(\frakE) $. If we complete $f$ and $g$ to triangles above, respectively below, we obtain a
diagram \[ \xymatrix{ Z \ar[r]  & X\ar[r]^f  & A \ar[r] \ar[d]^{\phi} & Z[1]  \\
 Y \ar[r]& T\ar[r]  & B[1] \ar[r]_{g[1]}  & Y[1] 
 .}
 \] 
Therefore, $g[1]\phi f=0$ if and only if there exists a homomorphism $X\to T$ such that the square
\[ \xymatrix{  X\ar[r]^f \ar@{-->}[d] & A  \ar[d]^{\phi}   \\
  T\ar[r]  & B  [1]
 }
 \]
 is commutative.
 \end{proof}

 Let $\CM$ be a class of maps in $\CA$. We define
 $$\CM^{\perp}=\{g\in\CA^\to \mid m\perp g \text{ for all } m\in \CM\},$$ 
  $${^{\perp}\CM}=\{g\in\CA^\to \mid g\perp m \text{ for all } m\in \CM\}.$$

  The proof of the next lemma is straightforward:
  
 \begin{lemma}
  Let $\CM$ be a class of homomorphisms in $\CA$. Then
  \begin{enumerate}[{\rm (1)}]
 \item $\CM^{\perp}$ and ${^{\perp}\CM}$ are ideals in $\CA$.
 \item $\CM^\perp[n]=\left(\CM[n]\right)^\perp$ and ${^\perp}{\CM}[n]={^\perp}{\left(\CM[n]\right)}$ for all $n\in\Z$, 
 where the ideals $\left(\CM[n]\right)^\perp$ and ${^\perp}{\left(\CM[n]\right)}$ of $\CA[n]$ are computed with respect to 
 $\frakE[n]$.  
\end{enumerate}
   \end{lemma}



\subsection{Special precovers and special preenvelopes}

If $\CI$ is an ideal in $\CA$, a homomorphism $i:X\to A$ from $\CI$ is a \textsl{special $\CI$-precover (\wrt $\frakE$)} 
if in the corresponding triangle $$B\to X\overset{i}\to A\overset{k}\to B[1]$$ we have $k\in (\CI^\perp)[1] \Ph(\frakE)$, 
i.e. $k=j[1]\varphi$ for some $j\in\CA^\to$ with $j\Ph(\frakE)\CI=0$ and some $\varphi\in\Ph(\frakE)$. 
We say that $\CI$ is a \textsl{special precovering ideal} if every object $A$ in $\CT$ has a special $\CI$-precover.

Dually, if $\CJ$ is an ideal in $\CT$, a homomorphism $j:B\to Y$ from $\CJ$ is a \textsl{special $\CJ$-preenvelope \wrt $\frakE$}
if in the corresponding triangle $$B\overset{j}\to Y\overset{\ell}\to A\overset{\psi}\to B[1]$$ we have 
$\psi\in \Ph(\frakE)\,({^\perp{ \CJ}})$, i.e. $\psi=\varphi i$ with $i\in \CA^\to$, 
$\varphi\in\Ph(\frakE)$ such that $\CJ[1]\Ph(\frakE) i=0$.
We say that $\CJ$ is a \textsl{special preenveloping ideal} if every object $A$ in $\CA$ has a special $\CJ$-preenvelope.

\begin{remark}
A homomorphism $i:X\to A$ is a \textsl{special $\CI$-precover \wrt $\frakE$} if 
there exists a homotopy pushout diagram  
\[ \tag{SPC} \ \xymatrix{ Y \ar[r]\ar[d]_j  & Z\ar[r]\ar[d]  & A \ar[r]^\phi \ar@{=}[d] & Y[1]\ar[d]^{j[1]}  \\
 B \ar[r]& X\ar[r]^i  & A \ar[r] ^\psi & B[1] 
 } \]
such that $j\in \CI^{\perp}$, and the top triangle is in $\frakE$.

Dually, a homomorphism $j:B\to Y$ is a \textsl{special $\CJ$-preenvelope \wrt $\frakE$}  
if there exists a homotopy pullback diagram 
\[ \tag{SPE}\  \xymatrix{ B \ar[r]^j\ar@{=}[d]  & Y\ar[r]\ar[d]  & A \ar[r]^\psi \ar[d]_i & B[1]\ar@{=}[d]  \\
 B \ar[r]& Z\ar[r]  & X \ar[r]^\phi  & B[1] 
 }  \]
such that $i\in {^{\perp}}\CJ$, and the bottom triangle is in $\frakE$.
\end{remark}

Observe that in both diagrams (SPC) and (SPE) all horizontal triangles are in $\frakE$ 
(we have automatically $\psi\in\Ph(\frakE) $), hence every special $\CI$-precover ($\CJ$-preenvelope) 
is an $\frakE$-deflation ($\frakE$-inflation). 

Moreover, we have $\psi \CI=0$ in (SPC), respectively $\CJ[1] \psi=0$ in (SPE). 
We may see that the terminology of special precover or preenvelope is justified in the sense of the following:

\begin{lemma}
Let $\CI$ and $\CJ$ be ideals. 
\begin{enumerate}[{\rm (1)}]
\item Every special $\CI$-precover \wrt $\frakE$ is an $\CI$-precover. 
\item Every special $\CJ$-preenvelope \wrt $\frakE$ is a $\CJ$-preenvelope.
\end{enumerate}
\end{lemma}

\begin{proof} If $i':X\to A$ is a map in $\CI$ then in (SPC) we have $\psi i=j[1]\phi i'=0$. Consequently 
$i'$ has to factor through $i$. 
\end{proof}

The role of special precovers and special preenvelopes is exhibited by the following 
version of Salce's Lemma, \cite[Lemma 5.20]{Go-Tr}.


\begin{theorem}\label{salce-lemma} (Salce's Lemma)
Let $\CI$ and $\CJ$ be ideals in $\CA$.
 \begin{enumerate}[{\rm (1)}]
  \item If there are enough $\frakE$-injective homomorphisms and $\CI$ is a precovering ideal, 
 then 
     $\CI^{\perp}$ is a special preenveloping ideal.  
    \item If there are enough $\frakE$-projective homomorphisms and $\CJ$ is a preenveloping ideal, 
 then 
  ${^{\perp}\CJ}$ is a special precovering ideal.  
 \end{enumerate}
\end{theorem}

\begin{proof} It is enough to prove (1). 

Consider $A\in\CA$ and let \[\tag{IE} A\overset{e}\to E\to X\overset{\psi}\to A[1]\] be a triangle 
such that $e$ is an $\frakE$-inflation which is $\frakE$-injective. 

Since $\CI$ is precovering for $\CA$ there exists an $\CI$-precover $i:I\to X$. By cobase change of the triangle (IE) along $i$ we get
the commutative diagram 
\[\xymatrix{ 
   A\ar[r]^{a}\ar@{=}[d] & Y\ar[d]\ar[r]    & I\ar[d]^{i} \ar[r]& A[1]\ar@{=}[d] \\
   A\ar[r]^{e}           & E\ar[r]    & X          \ar[r]^{\psi}& A[1].
}\]
We claim that $a$ is a special $\CI^{\perp}$-preenvelope of $A$. 
In order to prove this, it is enough to show that $a\in\CI^{\perp}$ since from the obvious inclusion $\CI\subseteq{^{\perp}(\CI^{\perp})}$  we know that $i\in {^{\perp}(\CI^{\perp})}$. 

In the subcategory $\CA[-1]$, the homomorphism $i[-1]:I[-1]\to X[-1]$ is an $\CI[-1]$-precover for $X[-1]$, and we have  
the solid part of the following commutative diagram:
\[\xymatrix{ Y[-1]\ar@{-->}[r]^{\kappa[-1]} \ar@{..>}[d]^{\eta} & B[-1] \ar@{-->}[d]^{\varphi[-1]} \ar@{..>}[ddl]_<<<<{\zeta} & &  \\            
  I[-1]\ar[r]_{\ \ \ \ \ \ u}\ar[d]_{i[-1]} & A\ar[r]^{a}\ar@{=}[d] & J\ar[d]\ar[r]    & I\ar[d]^{i} \\
  X[-1]\ar[r]^{-\psi[-1]} & A\ar[r]^{e}           & E\ar[r]    & X          
}\]

Let $\kappa:Y\to B$ be a map from $\CI$ and let $\varphi:B\to A[1]$ be an $\frakE$-phantom. Since $e$ is $\frakE$-injective
we have $e\varphi[-1]=0$, hence we can find a map $\zeta:B[-1]\to X[-1]$ such that $\varphi[-1]=-\psi[-1]\zeta$. Since $\CI[-1]$ 
is an ideal in $\CA[-1]$, we have $\zeta\kappa[-1]\in\CI[-1]$, hence $\zeta\kappa[-1]$ factorizes through the $\CI[-1]$-precover $i[-1]$. 
Therefore $\zeta\kappa[-1]=i[-1]\eta$ for some 
$\eta:Y[-1]\to I[-1]$. Finally $$a\varphi[-1]\kappa[-1]=-a\psi[-1]\zeta\kappa[-1]=
-a\psi[-1] i[-1] \eta=au\eta=0,$$ 
hence $a[1]\varphi\kappa=0$, and the proof is complete.
\end{proof}

From the proof of Theorem \ref{salce-lemma} we obtain the following corollary which will be useful in 
Section \ref{product-and-toda}.

\begin{corollary}\label{salce-lemma-cor}
{\rm (1)} If $\CI$ is an ideal and \[\xymatrix{ 
   A\ar[r]^{a}\ar@{=}[d] & Y\ar[d]\ar[r]    & I\ar[d]^{i} \ar[r]& A[1]\ar@{=}[d] \\
   A\ar[r]^{e}           & E\ar[r]    & X          \ar[r]& A[1]
}\] is a commutative diagram in $\CA$ such that the horizontal lines are triangles in $\frakE$, the homomorphism $e$ is  
$\frakE$-injective and $i$ is an $\CI$-precover for $X$, then the homomorphism $a$ is a special 
$\CI^\perp$-preenvelope for $A$.

{\rm (2)} If $\CJ$ is an ideal and \[\xymatrix{ 
 X\ar[d]^{j}\ar[r]    & P\ar[d] \ar[r]^p& A\ar@{=}[d] \ar[r] & X[1] \ar[d]^{j[1]} \\
 J\ar[r]    & Y          \ar[r]^{b} & A \ar[r] & J[1]
}\] is a commutative diagram in $\CA$ such that the horizontal lines are triangles in $\frakE$, $p$ is an 
$\frakE$-projective map and $j$ is a $\CJ$-preenvelope for $X$, then $b$ is a special 
${^\perp\CJ}$-precover.
\end{corollary}


\begin{remark}\label{diferenta-obiecte}
We want to point out that Theorem \ref{salce-lemma} shows us an important difference between orthogonal ideals and orthogonal 
classes of objects (or equivalently orthogonal object ideals). Let us suppose that $\CA$ has direct products and there are enough $\frakE$-projective 
homomorphisms. If we start with an object $A\in\CA$ then the class $\mathrm{Prod}(A)$ of all direct summands in direct products of copies
of $A$ is preenveloping, so the ideal $\Ideal{\mathrm{Prod}(A)}$ is also preenveloping. Therefore, the ideal 
$^\perp\Ideal{\mathrm{Prod}(A)}$ is precovering. On the other case, if we look at the category $\CA=\Ab$ as in 
Example \ref{exemplul-abelian} (here $\frakE$ is the canonical exact structure in $\Ab$) the class of all abelian groups $X$ such that 
$\Ext(X,\Ideal{\mathrm{Prod}(\Z)})=0$ is not necessarily precovering, as it is proved in \cite[Theorem 0.4]{ek-sh-03}.     
\end{remark}





\subsection{Ideal cotorsion pairs}

A pair of ideals $(\CI,\CJ)$ from $\CA$ is \textsl{orthogonal} if $i\perp j$ for all $i\in \CI$ and $j\in \CJ$, 
i.e. $\CJ\subseteq \CI^\perp$ and $\CI\subseteq {^\perp \CJ}$.

An \textsl{ideal cotorsion-pair} (with respect to  $\frakE$) is a pair of ideals
$(\CI, \CJ)$ in $\CA$ such that $\CJ=\CI^\perp$ and $\CI=\, {^\perp \CJ}$. 
 The ideal cotorsion pair $(\CI,\CJ)$ is \textsl{complete} if $\CI$ is a special precovering ideal and $\CJ$ is a special 
preenveloping ideal. 

\begin{theorem}\label{cotorsion-precovering}
{\rm (1)} If $\CI$ is a special precovering ideal then $(\CI,\CI^{\perp})$ is an ideal cotorsion pair.
Moreover, if there are enough $\frakE$-injective homomorphisms then the ideal cotorsion pair
$(\CI,\CI^\perp)$ is complete.

{\rm (2)} Dually, if $\CJ$ is a special preenveloping ideal then $({^\perp \CJ},\CJ)$ is an ideal cotorsion pair.
Moreover, if there are enough $\frakE$-projective homomorphisms then the ideal cotorsion pair
$({^\perp \CJ},\CJ)$ is complete.

\end{theorem}

\begin{proof} 
We have to show that $\CI={^{\perp}(\CI^{\perp})}$. The inclusion $\CI\subseteq {^{\perp}(\CI^{\perp})}$ 
is obvious.

Let $i':X'\to A$ be a homomorphism from ${^{\perp}(\CI^{\perp})}$. Since $\CI$ is special precovering 
we can find a triangle $Y\to X\overset{i}\to A\overset{k}\to Y[1]$ such that $i$ is a special $\CI$-precover for $A$. 
Then $k= j[1]\phi$ for some $j\in \CI^{\perp}$ and some $\phi\in\Ph(\frakE)$.  
All these data are represented in the solid part of the following commutative
diagram: \[ \xymatrix{T\ar[r]\ar[d] & Z\ar[r]\ar[d] & A\ar@{=}[d]\ar[r]^{\phi}& T[1]\ar[d]^{j[1]}\\
	       Y\ar[r] & X\ar[r]^{i}      & A\ar[r]^{k}      & Y[1]. \\ 
				& & X' \ar[u]_{i'} \ar@{-->}[ul]^{g}& } \]

Because $i'\perp j$ we obtain 
$ki'=j[1]\phi i'=0$, so $i'$ factors through the weak kernel $i$ of $k$, i.e. $i'=ig$ for some $g:X'\to X$.
Therefore $i'\in\CI$, and the proof for the first statement is complete. 

The second statement follows from Salce's lemma. 
\end{proof}

{
\begin{example}\label{trivial-case-special}\label{ex-trivial}
If $\CA=\CT$ and $\frakE=\frakD$ is the class of all triangles then every precovering ideal is special since 
every triangle $B\to X\to A\overset{\psi}\to B[1]$ can be embedded in a commutative diagram
\[  \xymatrix{ A[-1] \ar[r]\ar[d]_{\psi[-1]}  & 0\ar[r]\ar[d]  & A \ar@{=}[r] \ar@{=}[d] & A\ar[d]^{\psi}  \\
 B \ar[r]& X\ar[r]^i  & A \ar[r] ^\psi & B[1] .
 } \]
Dually, every preenveloping ideal in $\CT$ is special with respect to the class $\frakD$ of all triangles in $\CT$.

Therefore, for every precovering ideal $\CI$  we obtain that $(\CI,\CI^\perp)$ is a complete ideal cotorsion pair, 
hence $\CI^\perp$ is 
a preenveloping ideal. It follows that \begin{itemize} \item[$(*)$] for every $A\in\CT$ 
there is a triangle 
\[\mathfrak{d}_A: X_A\overset{i_A}\to A\overset{j_A}\to Y_A\to X_A[1],\] with $i_A\in\CI$ and $j_A\in\CI^\perp$.\end{itemize} 
Conversely, 
a pair $(\CI,\CJ)$ of ideals in $\CT$ is an ideal cotorsion pair (with respect to $\frakD$) if and only if it
has the property $(*)$, 
where $\CI^\perp$ is replaced by $\CJ$. 

For instance, if   $(\X,\Y)$ is a (co)torsion theory in $\CT$ (in particular a t-structure), the pair $$(\Ideal\X, \Ideal\Y)$$ 
is a complete ideal cotorsion pair 
with respect to the proper class $\frakD$ of all triangles in $\CT$. 

\end{example}
}

\section{Products of ideals and Toda brackets}\label{product-and-toda}

In this section we continue to fix an extension closed subcategory $\CA$ of $\CT$, but in this section 
\textsl{we will assume that  $\frakE$ is an exact structure in $\CA$}.

\subsection{Toda brackets}

In the following we will use the algebraic concept of Toda bracket as it is defined in \cite{Sto-snake}. This concept let us 
to
generalize the operations $\diamond$ introduced in \cite{rouq} for (object ideals in) triangulated categories (cf. Proposition 
\ref{lem-toda-obiecte}) and in \cite{Fu-Herzog} for exact categories (cf. \cite[Lemma 6]{Fu-Herzog}).

Let $$\mathfrak{d}:\ Y\overset{f}\to Z\overset{g}\to X\overset{\varphi}\to Y[1]$$ be a triangle in $\CT$. 
If $i:Y\to U$ and $j:V\to Z$ are two homomorphisms then the \textsl{Toda bracket} $\langle i,j\rangle_{\mathfrak{d}}$ 
is the set of all homomorphisms $\zeta:V\to U$ such that $\zeta=\zeta'\zeta''$, where $\zeta'':V\to Z$ and 
$\zeta':Z\to U$ are homomorphisms which make the diagram 
 \[\tag{TB}\xymatrix{& &  & V\ar@{-->}[dl]_{\zeta''}\ar[d]^{j} & \\ 
            \mathfrak{d}:&Y\ar[r]^{f}\ar[d]_{i} & Z\ar[r]^{g}\ar@{-->}[dl]^{\zeta'} & X\ar[r] & Y[1]\\
            & U &  &  & 
            }\]
commutative. 

If $\CI$ and $\CJ$ are two classes of homomorphisms then the union of all Toda brackets $\langle i,j\rangle_{\mathfrak{d}}$ 
with $i\in \CI$, $j\in \CJ$ and $\mathfrak{d}\in\frakE$ is denoted by $\langle \CI,\CJ\rangle_\frakE$, and it is called 
the \textsl{Toda bracket of $\CI$ and $\CJ$ induced by $\frakE$}.

\begin{remark}
Let $i$ and $j$ be two homomorphisms and let $\mathfrak{d}$ be a triangle in $\CT$. Then
 $\langle i,j\rangle_\mathfrak{d}\neq \varnothing$ if and only if $i$ is injective relative to $\mathfrak{d}$ and $j$ is projective relative to $\mathfrak{d}$.
\end{remark}

{%
\begin{remark}\label{dualizare}
Let us consider the dual category $\CT^\star$, and we denote by $\CI^\star$ and $\frakE^\star$ the ideal, respectively the almost exact structure
induced by $\CI$ and $\frakE$ in $\CT^\star$. Then for every two ideals $\CI$ and $\CJ$ in $\CT$ we have
$\langle \CI^\star,\CJ^\star\rangle_{\frakE^\star}=(\langle \CJ,\CI\rangle_\frakE)^\star$.  
\end{remark}

\begin{lemma}
If $\CI$ and $\CJ$ are ideals in $\CA$ then  $\langle \CI,\CJ\rangle_\frakE$ is also an ideal in $\CA$. 
\end{lemma}

\begin{proof}
 It is straightforward to check that $0\in\langle 0,0\rangle_\frakE\subseteq \langle \CI,\CJ\rangle_\frakE$, and that $\langle \CI,\CJ\rangle_\frakE$ is closed with respect to compositions 
 with arbitrary maps and finite direct sums. 
\end{proof}

For further references, let us consider the following remark which can be extracted from \cite[Lemma 6]{Fu-Herzog}. 

\begin{lemma}\label{remark-toda-ext}
If $\CI$ and $\CJ$ are ideals and $\xi:V\to U$ is a homomorphism in $\langle \CI,\CJ\rangle_\frakE$ then there exists a commutative diagram
\[\xymatrix{ 
&Y\ar[r]\ar[d]_{i} & P\ar[r]\ar[d]^{\zeta} \ar@/^/ @{<--}[r]^{\alpha}& V\ar[r]^{0} \ar[d]^{j}& Y[1]\ar[d]_{i[1]}\\
						&U\ar[r]\ar@/_/ @{<--}[r]_{\beta}& Q\ar[r] & X\ar[r]^{0} & U[1]\\
                        }\]
such that the horizontal lines are splitting triangles, $i\in \CI$ and $j\in \CJ$, the homomorphisms $\alpha:V\to P$ and $\beta:Q\to U$  
are partial inverses for $P\to V$ respectively $U\to Q$ and $\xi=\beta\zeta\alpha$.
\end{lemma}

\begin{proof} 
Starting with the diagram $(\mathrm{TB})$ we can construct via a base change and a cobase change the following 
commutative diagram 
\[\xymatrix{& &  & V\ar@{-->}[dl]_{\alpha}\ar@{-->}[ddl]^{\zeta''}\ar@{=}[d] & \\ 
\mathfrak{d}j:&Y\ar[r]\ar@{=}[d] & P\ar[r]\ar[d] & V\ar[r] \ar[d]^{j}& Y[1]\ar@{=}[d]\\
            \mathfrak{d}:&Y\ar[r]^{f}\ar[d]_{i} & Z\ar[r]^{g}\ar[d]\ar@{-->}[ddl]_{\zeta'} & X\ar[r]\ar@{=}[d] & Y[1]\ar[d]_{i[1]}\\
						i\mathfrak{d}:&U\ar[r]\ar@{=}[d] & Q\ar[r]\ar@{-->}[dl]^{\beta} & X\ar[r] & U[1]\\
            & U &  &  &\ \ . 
            }\]
The homomorphisms $\alpha$ and $\beta$ are constructed via the weak universal property 
of the homotopy pullback and pushout. Now the conclusion is obvious.
\end{proof}

}

For further applications, let us study Toda brackets associated to object ideals.  

\begin{proposition}\label{lem-toda-obiecte}
Let $\CA$ be an extension closed full subcategory of $\CT$ and $\frakE$ an almost exact structure from $\CA$. 
If $\CP$ and $\CQ$ are
 two classes of objects in $\CA$ closed under finite direct sums, and $\mathcal{V}$ is the class of all objects 
 $V$ which lie in triangles 
$\mathfrak{d}:\ Q\to V\to P\to Q[1]$ 
with $\mathfrak{d}\in\frakE$, $P\in\CP$ and $Q\in\CQ$ then \begin{enumerate}[{\rm (a)}] \item $\CV$ is closed under finite direct sums;
\item $\langle \Ideal\CQ,\Ideal\CP\rangle_\frakE= \Ideal\CV$.
\end{enumerate}
\end{proposition}
\begin{proof}
(a) is a simple exercise.

(b)
If $\zeta:W\to U$ is in $\langle \Ideal\CQ,\Ideal\CP\rangle_\frakE$ then we have a diagram 
\[\xymatrix{& &  & W\ar@{-->}[dl]_{v}\ar@{-->}[ddl]^{\zeta''}\ar[d]^{\pi} & \\ 
\mathfrak{d}j:&Y\ar[r]^{f'}\ar@{=}[d] & A\ar[r]\ar[d]_{\alpha} & P\ar[r] \ar[d]^{j}& Y[1]\ar@{=}[d]\\
            \mathfrak{d}:&Y\ar[r]^{f}\ar[d]_{i} & Z\ar[r]^{g}\ar@{-->}[ddl]^{\zeta'} & X\ar[r] & Y[1]\\
						&Q\ar[d]_{\rho} & & & \\
            & U &  &  &\ \ 
            }\]
such that $P\in\CP$, $Q\in \CQ$, $\mathfrak{d}\in\frakE$ and $\zeta=\zeta'\zeta''=\zeta'\alpha v$. 

Then we construct via a homotopy pushout along $i$ a triangle $i\mathfrak{d}j$, 
hence we have a commutative diagram 
\[\xymatrix{ \mathfrak{d}:&Y\ar[r]^{f}\ar@{=}[d] & Z\ar[r]^{g}\ar@{<-}[d]^{\alpha}\ar[dddl]^{\zeta'} & X\ar[r]\ar@{<-}[d]^{j} & Y[1]\ar@{=}[d]\\
\mathfrak{d}j:&Y\ar[r]^{f'}\ar[d]_{i} & A\ar[r]\ar[d]^{\alpha'} & P\ar[r] \ar@{=}[d]& Y[1]\ar[d]^{i[1]}\\
            i\mathfrak{d}j:&Q\ar[r]\ar[d]_{\rho} & V\ar[r] \ar@{-->}[dl]^{\xi}& P\ar[r] & Q[1]\\
						&U & & & .
						}\]
Since $\zeta'\alpha f'=\zeta'f=\rho i$, there exists a homomorphism $\xi:V\to U$ such that $\zeta'\alpha=\xi\alpha'$.
Therefore $\zeta=\zeta'\alpha v=\xi\alpha'v$, and it follows that $\zeta$ factorizes through $V$.
Then $\zeta\in \Ideal\CV$. 

Conversely, if we have a triangle $\mathfrak{d}:\ Q\to V\to P\to Q[1]$ in $\frakE$ 
with $P\in\CP$ and $Q\in\CQ$ then we can construct the commutative diagram
\[\xymatrix{& &  & V\ar@{=}[dl]\ar[d] & \\ 
            \mathfrak{d}:&Q\ar[r]\ar[d] & V\ar[r]\ar@{=}[dl] & P\ar[r] & Q[1],\\
            & V &  &  & \ \ 
            }\]
hence $V$ is an object in the ideal $\langle \Ideal\CQ,\Ideal\CP\rangle_\frakE$.
\end{proof}

\subsection{Wakamatsu's Lemma}
We will prove here an ideal version for Wakamatsu's Lemma which generalizes the 
corresponding results proved in \cite[Lemma 37]{Fu-Herzog} for exact categories and in 
\cite[Lemma 2.1]{Jorge-AR} for object ideals in triangulated categories.  Let $\CI$ be an 
ideal in $\CA$. An $\CI$-precover $i:Z\to A$ 
is an \textsl{$\CI$-cover} if it is an $\frakD_\CA$-deflation and for 
every endomorphism $\alpha$ of $Z$ from $i\alpha=i$ it that follows 
$\alpha$ is an isomorphism. 
We note that there are categories when every precovering (preenveloping) ideal is covering 
(enveloping), e.g. the category of finitely 
generated modules over artin algebras, cf. \cite[Proposition 1.1]{Aus-adv} or in the case of 
$k$-linear $\Hom$-finite triangulated categories, \cite[Lemma 1.1]{Buan-waka}.

\begin{lemma}\label{waka} Let $\CI$ be an ideal in $\CA$ which is closed under Toda brackets, 
that is $\langle\CI,\CI\rangle_\frakE\subseteq\CI$, 
and let $i:Z\to A$ be an $\CI$-cover for $A$.
If $$K\overset{\kappa}\to Z\overset{i}\to A\overset{\nu}\to K[1]$$ is the corresponding triangle then $1_K\in \CI^\perp$.
\end{lemma} 

\begin{proof}
We have to prove that for every $\varphi:Y\to K[1]$ from $\Ph(\frakE)$ and every $i':X\to Y$ from $\CI$ we have $\varphi i'=0$.
 
Let $\varphi\in \Ph(\frakE)$ and $i'\in \CI$ as before. Using homotopy pullbacks along $\varphi$ and $i'$ we obtain the 
solid part of following commutative diagram 
\[\xymatrix{  Z\ar[r]^{\upsilon}\ar@{=}[d] & U\ar[d]^{\beta}\ar[r]\ar@{-->}[ddl]  \ar@/^/ @{<--}[r]  & X\ar[r]^{\eta}\ar[d]^{i'}   & Z[1]\ar@{=}[d]\\
Z\ar[r]\ar@{=}[d] & T\ar[d]^{\alpha}\ar[r]    & Y\ar[r]^{\psi}\ar[d]^{\varphi}   & Z[1]\ar@{=}[d]\\
  Z\ar[r]^{i}           & A\ar[r]^{\nu}    & K[1]\ar[r]& Z[1]          
.}\]
Since $\varphi\in\Ph(\frakE)$ we obtain $\psi\in \Ph(\frakE)$, hence the triangle $Z\to T\to Y\to Z[1]$ is in $\frakE$. Moreover,
the composition $U\to X\overset{i'}\to Y$ is in $\CI$, hence $\alpha\beta\in\langle \CI,\CI\rangle_\frakE\subseteq \CI$. It follows that
$\alpha\beta$ factorizes through $i$, hence we can find a homomorphism $\gamma:U\to Z$ such that $\alpha\beta=i\gamma$. Then 
$i=\alpha\beta\upsilon=i\gamma\upsilon$, and it follows that $\gamma\upsilon$ is an automorphism of $Z$.   
Since $\gamma\upsilon\eta[-1]=0$ we obtain $\eta=0$. Then the top triangle splits, and it follows that $\varphi i'$ factorizes 
through $\nu\alpha\beta=\nu i\gamma=0$. Then $\varphi i'=0$, and the proof is complete. 
\end{proof}

{%
Now we can apply the previous results to obtain the object version of Wakamatsu's Lemma. In the case $\frakE=\frakD$ this was proved in 
 \cite[Lemma 2.1]{Jorge-AR}. 

\begin{corollary}\label{cor-waka}
Let $\CX$ be a class of objects in $\CT$.
If $\CX$ is closed with respect to $\frakE$-extensions, and $$K\to X\overset{i}\to A\to K[1]$$ is an $\frakE$-triangle such that $i$ is an 
$\CX$-cover then $\Hom(\CX,K[1])\cap \Ph(\frakE)=0$. 
\end{corollary}

\begin{proof}
Let $\CV$ be the class of all objects $V$ which lie in  $\frakE$-conflations $X\to V\to X'\to X[1]$ with $X,X'\in \CX$. 
Applying Proposition \ref{lem-toda-obiecte} and the hypothesis we have 
$\langle \Ideal\CX,\Ideal\CX\rangle_\frakE=\Ideal\CV\subseteq \Ideal\CX$, hence $\Ideal\CX$ is closed with respect to Toda brackets. Then the 
conclusion follows from Lemma \ref{waka}. 
\end{proof}
}


\subsection{Products of ideals}
It is easy to see (as in the proof of Theorem \ref{salce-lemma}) that if $\CI$ and $\CJ$ are ideals, $i:I\to A$ is an $\CI$-precover for $A$ and $j:J\to I$ is a $\CJ$-precover 
for $I$ then $ij$ is an $\CI\CJ$-precover for $A$. Therefore, if $\CI$ and $\CJ$ are precovering ideals then $\CI\CJ$ is also precovering, see \cite[Lemma 3.6]{Meyer2}. 



The main aim of this subsection 
is to prove that if $\CI$ and $\CJ$ are special 
precover ideals (with respect to $\frakE$) then $\CI\CJ$ is also special precovering, 
and to compute $(\CI\CJ)^\perp$.

\begin{lemma}\label{Toda-inc}
If $\CI$ and $\CJ$ are ideals in $\CA$ then $\langle \CJ^\perp,\CI^\perp\rangle_\frakE\subseteq (\CI\CJ)^{\perp}$.
\end{lemma}

\begin{proof}
Let $\zeta=\zeta'\zeta''\in \langle \CJ^\perp,\CI^\perp\rangle_\frakE$.
 In order to prove that $\zeta'\zeta''\in(\CI\CJ)^\perp$ we consider a chain of composable homomorphisms 
$U\overset{j}\to T\overset{i}\to W\overset{\phi}\to V$  such that $i\in\CI$, 
 $j\in\CJ$ and $\phi\in\Ph(\frakE)$. 
We have the solid part of the following commutative diagram
 \[\xymatrix{ U[-1]\ar[r]^{j[-1]} & T[-1]\ar[r]^{i[-1]}\ar@{-->}[dd]^{\phi'[-1]} & W[-1]\ar[dr]^{\phi[-1]} & & \\
 & & & V\ar[d]^{\nu}\ar[dl]_{\zeta''} & \\ 
     &       Y\ar[r]^{f}\ar[d]^{\mu} & Z\ar[r]^{g}\ar[dl]^{\zeta'} & X\ar[r] & Y[1]\\
        &    U &  & & \ \ \ ,
            }\]
where the row $Y\to Z\to X\to Y[1]$ is a triangle in $\frakE$, $\mu\in \CJ^\perp$ and $\nu\in\CI^\perp$.

Then $g\zeta''\phi[-1] i[-1]=\nu \phi[-1] i[-1]=0$ since $\nu \in\CI^\perp$. Therefore $\zeta''\phi[-1] i[-1]$ factors 
 through $f$, i.e. there exists a homomorphism $\phi'[-1]:T[-1]\to Y$ such that $f\phi'[-1]=\zeta''\phi[-1] i[-1]$. 
 We observe that $f[1]\phi'$ factors through $\phi$ hence $f[1]\phi'\in\Ph(\frakE)$. 
Since $f$ is an $\frakE$-inflation, the 
 saturation of $\frakE$ implies $\phi'\in\Ph(\frakE)$. 
 Finally we have: 
 \[(\zeta'\zeta'')[1]\phi ij=\zeta'[1] f[1]\phi'j=\mu[1]\phi'j=0\]
since $\mu\in\CJ^\perp$. 
\end{proof}

{
\begin{corollary}\label{perp-de-idempotent}
If $\CI$ is an idempotent ideal then $\CI^\perp$ is closed with respect to Toda brakets.
\end{corollary}
}

\begin{corollary}\label{cor-toda-inc}
If $\CI$ is an ideal in $\CA$ then $\langle \CI^\perp,\frakE\text{{\rm -inj}}\rangle_\frakE \subseteq \CI^\perp$.  
\end{corollary}

\begin{proof}
Applying Lemma \ref{Toda-inc} we have  
$$\langle \CI^\perp,\frakE\text{{\rm -inj}}\rangle_\frakE= \langle \CI^\perp,(\CA^\to)^\perp\rangle_\frakE 
\subseteq (\CA^\to \CI)^\perp=\CI^\perp,$$
and the proof is complete.
\end{proof}

\begin{theorem}\label{chain-rule}
 Let $\CI$ and $\CJ$ be two special precovering ideals in $\CA$. Then the product ideal $\CI\CJ$ is also special precovering. 

If $A\in\CA$, $i:I\to A$ is a special $\CI$-precover, and $j:J\to I$ is a special $\CJ$-precover then 
 $ij:J\to A$ is a special $\CI\CJ$-precover. Moreover, $ij$ can be embedded in a homotopy pushout diagram
 \[\xymatrix{
 Z''\ar[r]\ar[d]^{\zeta} & J''\ar[d]\ar[r] & A\ar@{=}[d] \ar[r]& Z''[1]\ar[d]\\ 
 Z\ar[r]        &    J\ar[r]^{ij} & A \ar[r] & Z[1]
            }\]
with $\zeta\in \langle \CJ^\perp,\CI^\perp\rangle_\frakE $. 
\end{theorem}

\begin{proof}
 Consider the diagrams 
 \[\xymatrix{X'\ar[r]\ar[d]^{\xi} & I'\ar[r]\ar[d] & A\ar[r]\ar@{=}[d] & X'[1]\ar[d]^{\xi[1]}\\
             X\ar[r]              & I\ar[r]^{i}    & A\ar[r]           & X[1]} \]
  and    
  \[\tag{$\sharp$}\  \xymatrix{Y'\ar[r]\ar[d]^{\upsilon} & J'\ar[r]\ar[d] & I\ar[r]\ar@{=}[d] & Y'[1]\ar[d]^{\upsilon[1]}\\
             Y\ar[r]                    & J\ar[r]^{j}    & I\ar[r]           & Y[1]} \] 
						with $\xi\in\CI^\perp$ and $\upsilon\in\CJ^\perp$, which emphasise the facts that $i$ and $j$ are special precovers.  
By pulling back along $I'\to I$ we obtain the commutative diagram 

 \[\xymatrix{Y'\ar[r]\ar@{=}[d] & J''\ar[r]\ar[d] & I'\ar[r]\ar[d] & Y'[1]\ar@{=}[d]\\
             Y'\ar[r]                    & J'\ar[r]    & I\ar[r]           & Y'[1].} \]

Using the octahedral axiom, we extend these diagrams to the solid part of the following diagram:
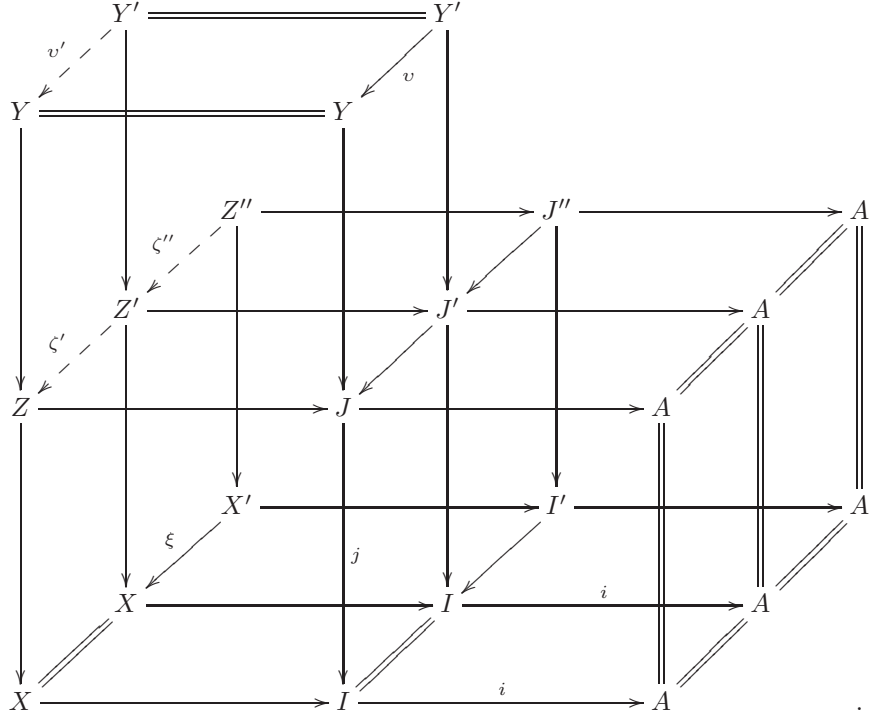
\[\xymatrix{ 
             &Y'\ar@{=}[rrr]\ar[ddd]\ar@{-->}[dl]_{\upsilon'} & & & Y'\ar[dl]^{\upsilon}\ar[ddd] &&&&   \\
             Y\ar@{=}[rrr]\ar[ddd] &&& Y\ar[ddd] &&&&& \\
             && Z''\ar[rrr]\ar@{-->}[dl]_{\zeta''}\ar[ddd] &&& J''\ar[rrr]\ar[dl]\ar[ddd] &&& A\ar@{=}[dl]\ar@{=}[ddd]\\
             & Z'\ar[rrr]\ar@{-->}[dl]_{\zeta'}\ar[ddd] &&& J'\ar[rrr]\ar[dl]\ar[ddd] &&& A\ar@{=}[dl]\ar@{=}[ddd] & \\
             Z\ar[rrr]\ar[ddd] &&& J\ar[rrr]\ar[ddd]^{j} &&& A\ar@{=}[ddd] && \\
             && X'\ar[rrr]\ar[dl]_{\xi} &&& I'\ar[rrr]\ar[dl] &&& A\ar@{=}[dl]\\
             & X\ar[rrr]\ar@{=}[dl] &&& I \ar[rrr]^{i}\ar@{=}[dl] &&& A\ar@{=}[dl] &\\
             X\ar[rrr] &&& I \ar[rrr]^{i} &&& A && .
}\]
Here all vertical and horizontal lines (from left to right) are triangles in $\frakE$ and all squares but the top horizontal square are
commutative. 

The homomorphism $\zeta''$ is constructed as follows: 
we have the equality $$(Z''\to X'\overset{\xi}\to X\to I)=(Z''\to J''\to J'\to I)$$ and $Z'$ is a homotopy pullback of the angle $X\to I\leftarrow J'$,
hence there exists a homomorphism $\zeta'':Z''\to Z'$ making the diagram commutative. The homomorphism $\zeta':Z'\to Z$ is obtained in an analogous way
by using the equality 
$$(Z'\to J'\to J\overset{j}\to I)=(Z'\to X\to I).$$ Finally, we 
consider a homomorphism $\upsilon':Y'\to Y$ such that $(\upsilon', \zeta',1_X)$ is a homomorphism of triangles. 


We have \begin{align*} (Y'\overset{\upsilon'} \to Y\to J)&=( Y'\overset{\upsilon'} \to Y\to Z\to J)
=(Y'\to Z'\overset{\zeta'} \to Z\to J)\\ &=(Y'\to Z'\to J'\to J),\end{align*} 
and the diagram $(\sharp)$ is obtained as a homotopy pushout diagram. Therefore $\upsilon'$ factorizes through $\upsilon$. Then $\upsilon'\in \CJ^{\perp}$.

We extract from the above diagram the following commutative diagram
 \[\xymatrix{ 
 & & Z''\ar[r]\ar[d]^{\zeta''} & X'\ar[d]^{\xi}\ & \\ 
     &       Y'\ar[r]^{f}\ar[d]^{\upsilon'} & Z'\ar[r]^{g}\ar[d]^{\zeta'} & X\ar[r] & Y'[1]\\
        &    Y\ar[r] & Z & & \ \ \ ,
            }\]
and using Lemma \ref{Toda-inc} we obtain 
$\zeta'\zeta''\in \langle \CJ^\perp,\CI^\perp\rangle_\frakE \subseteq (\CI\CJ)^{\perp}$.

From the commutative diagram
\[\xymatrix{
 Z''\ar[r]\ar[d]^{\zeta'\zeta''} & J''\ar[d]\ar[r] & A\ar@{=}[d] \ar[r]& Z''[1]\ar[d]\\ 
 Z\ar[r]        &    J\ar[r]^{ij} & A \ar[r] & Z[1]
            }\]
we obtain the conclusions stated in theorem.
\end{proof}

\begin{corollary}
 If $\CI$ is a special precovering ideal then the same is true for any ideal in the chain:
 \[\CI=\CI^1\supseteq\CI^2\supseteq\CI^3\supseteq\cdots.\]
\end{corollary}


\subsection{Ghost lemma}
In the following we need a result which generalizes Salce's Lemma (in the case $\frakE$ is an exact structure).

\begin{lemma}\label{lema-toda-env}
Let $\mathcal{K}$, $\mathcal{L}$ be ideals in $\CA$, and 
let \[\tag{IE} A\overset{e}\to E\to X\to A[1]\] be a triangle in $\frakE$ such that $e\in \mathcal{L}$.

Let $i:I\to X$ be a homomorphism which can be embedded in a 
commutative diagram \[ \tag{PO} \xymatrix{Y\ar[r]\ar[d]^{g} & Z\ar[d]^{h}\ar[r]& X\ar@{=}[d]\ar[r] & Y[1]\ar[d]^{g[1]}\\
	      W\ar[r]^{w}      & I\ar[r]^{i}      & X\ar[r]^{\phi}           & W[1]} \]
such that $g\in\mathcal{K}$ and the rows in this diagram are triangles in $\frakE$. If the 
diagram
\[\tag{PB} \xymatrix{  A\ar[r]^{a}\ar@{=}[d] & J\ar[d]^{\alpha}\ar[r]    & I\ar[r]\ar[d]^{i}   & A[1]\ar@{=}[d]\\
  A\ar[r]^{e}           & E\ar[r]    & X\ar[r]& A[1]          
}\]
is obtained as a homotopy pullback along $i$ then  
$$a\in \langle \mathcal{K}, \mathcal{L}\rangle_{\frakE}.$$ 
\end{lemma}

\begin{proof} 
	   
Using a cobase change of the triangle (IE) along $i$ we complete the diagram (PB) to the commutative diagram
\[\xymatrix{            & W\ar@{=}[r]\ar[d]& W\ar[d]^{w}         &               &\\
  A\ar[r]^{a}\ar@{=}[d] & J\ar[d]^{\alpha}\ar[r]    & I\ar[r]\ar[d]^{i}   & A[1]\ar@{=}[d]&\\
  A\ar[r]^{e}           & E\ar[r]\ar[d]    & X\ar[r]\ar[d]^{\phi}& A[1]          &\\
                        & W[1]\ar@{=}[r]   & W[1]                & 		& \ .
}\]
Moreover, using this time the homomorphism $ih$,  we can modify the diagram (PO) to obtain the following commutative diagram:
\[\xymatrix{              & Y\ar@{=}[r]\ar[d]& Y\ar[d]             &              & \\
  A\ar[r]^{f}\ar@{=}[d] & C\ar[d]\ar[r]    & Z\ar[r]\ar[d]^{ih}  & A[1]\ar@{=}[d]&\\
  A\ar[r]^{e}           & E\ar[r]\ar[d]    & X\ar[r]\ar[d]       & A[1]          &\\
                        & Y[1]\ar@{=}[r]   & Y[1]                & 		 .}
												\]
Note that in the above two diagrams all rows and columns are triangles 
in $\frakE$.
The horizontal  cartesian rectangle from the previous diagram can be obtained as a juxtaposition of two cartesian diagrams 
\[ \xymatrix{A\ar[r]^{f}\ar@{=}[d] & C\ar[d]^{k}\ar[r]& Z\ar[d]^{h}\ar[r] & A[1]\ar@{=}[d]\\
	      A\ar[r]^{a}\ar@{=}[d]      & J\ar[r]\ar[d]^{\alpha}      & I\ar[r]\ar[d]^{i}           & A[1]\ar@{=}[d]\\
				  A\ar[r]^{e}           & E\ar[r]   & X\ar[r]       & A[1]          &,} \] 
and using the octahedral axiom we complete the middle commutative square in the following diagram to a homomorphism of triangles:	      
\[ \xymatrix{Y\ar[r]\ar[d]^{g'} & C\ar[d]^{k}\ar[r]& E\ar@{=}[d]\ar[r] & Y[1]\ar[d]^{g'[1]}\\
	      W\ar[r]        & J\ar[r]^{\alpha}         & E\ar[r]           & W[1].} \]

Now denote $\delta=g-g':Y\to W$.
Since \begin{align*}
(Y\overset{g}\to W\overset{w}\to I)&= 
(Y\to Z\overset{h} \to I)=(Y\to C\to Z\overset{h}\to I) \\
&= (Y\to C\overset{k}\to J\to I)=(Y\overset{g'}\to W\to J\to I)\\ 
&= (Y\overset{g'}\to W\overset{w}\to I),
\end{align*} 
we obtain $w\delta=0$, hence $\delta$ factorizes through $\phi[-1]$. But $\phi[-1]$ factorizes through
$g$, and it follows that $g'$ factorizes through $g$. Therefore $g'\in \mathcal{K}$.
Using the commutative diagram 

\[ \xymatrix{ & & A\ar[d]^{e}\ar[dl]_{f}& \\
Y\ar[r]\ar[d]^{g'} & C\ar[d]^{k}\ar[r]& E\ar[r] & Y[1]\\
	      W\ar[r]        & J        &   & } \]
together with $a=kf$ we obtain $a\in \langle \mathcal{K}, \mathcal{L}\rangle_{\frakE}$.
\end{proof}


As a first application, we improve Corollary \ref{cor-toda-inc}.

\begin{corollary}\label{Toda-perp-cor-1}
Suppose that there are enough $\frakE$-injective homomorphisms. 
If $\CI$ is a special precovering ideal then $\CI^\perp=\langle \CI^\perp, \frakE\text{{\rm -inj}}\rangle_\frakE$.
\end{corollary}

\begin{proof}
Using Corollary \ref{salce-lemma-cor} we can construct for every object $A$ in $\CA$ a special 
$\CI^\perp$-preenvelope via a pullback diagram
\[\xymatrix{  A\ar[r]^{a}\ar@{=}[d] & K\ar[d]\ar[r]    & I\ar[r]\ar[d]^{i}   & A[1]\ar@{=}[d]\\
  A\ar[r]^{e}           & E\ar[r]    & X\ar[r]& A[1]          
}\]
such that $e$ is injective and $i$ is a special precover for $X$. 
By Lemma \ref{lema-toda-env} it follows that $a\in \langle \CI^\perp, \frakE\text{{\rm -inj}}\rangle_\frakE$, hence
$\CI^\perp\subseteq \langle \CI^\perp, \frakE\text{{\rm -inj}}\rangle_\frakE$. Using Corollary \ref{cor-toda-inc}
we obtain 
 $\CI^\perp= \langle \CI^\perp, \frakE\text{{\rm -inj}}\rangle_\frakE$.
\end{proof}

\begin{theorem}\label{perp-product}
Suppose that there are enough $\frakE$-injective homomorphisms. 
If $\CI$ and $\CJ$ are special precovering ideals in $\CA$ then 
$(\CI\CJ)^\perp=\langle \CJ^\perp, \CI^\perp\rangle_\frakE$. 
\end{theorem}

\begin{proof}
By Lemma \ref{Toda-inc}, we only have to prove the inclusion $(\CI\CJ)^\perp\subseteq \langle \CJ^\perp, \CI^\perp\rangle_\frakE$. Since
$\CI\CJ$ is a special precovering ideal, it follows that the ideal $(\CI\CJ)^\perp$ is special preenveloping. Therefore,
it is enough to prove that for every object $A$ in $\CA$ there exists a special $(\CI\CJ)^\perp$-preenvelope which belongs to
$\langle \CJ^\perp, \CI^\perp\rangle_\frakE$.

Let $A$ be an object in $\CA$. 
As in the proof of Corollary \ref{Toda-perp-cor-1},
we use Corollary \ref{salce-lemma-cor} to construct for every object $A$ in $\CA$ a special 
$(\CI\CJ)^\perp$-preenvelope $a:A\to K$ via a pullback diagram
\[ \xymatrix{  A\ar[r]^{a}\ar@{=}[d] & K\ar[d]\ar[r]    & J\ar[r]\ar[d]^{ij}   & A[1]\ar@{=}[d]\\
  A\ar[r]^{e}           & E\ar[r]    & X\ar[r]& A[1]          
}\]
such that $i:I\to X$ is a special $\CI$-precover for $X$ and $j:J\to I$ is a special $\CJ$-precover for $I$.
If we consider the homotopy pullback of the triangle $$A\to E\to X\to A[1]$$ 
along $i$, we can assume that the above commutative diagram
is constructed using two homotopy pullbacks as in the following commutative diagram 
\[\tag{PB'} \xymatrix{  A\ar[r]^{a}\ar@{=}[d] & K\ar[d]\ar[r]    & J\ar[r]\ar[d]^{j}   & A[1]\ar@{=}[d]\\
A\ar[r]^{b}\ar@{=}[d] & L\ar[d]\ar[r]    & I\ar[r]\ar[d]^{i}   & A[1]\ar@{=}[d]\\
  A\ar[r]^{e}           & E\ar[r]    & X\ar[r]& A[1]          
,}\]
where both horizontal rectangles are cartesian. 

By Corollary \ref{salce-lemma-cor} we have $b\in \CI^\perp$. Moreover, since $j$ is a 
special precover, it can be embedded in a commutative diagram 
\[ \xymatrix{Y\ar[r]\ar[d]^{g} & Z\ar[d]^{h}\ar[r]& X\ar@{=}[d]\ar[r] & Y[1]\ar[d]^{g[1]}\\
	      W\ar[r]^{w}      & J\ar[r]^{j}      & I\ar[r]^{\phi}           & W[1]} \]
such that $g\in \CJ^\perp$. Then we can apply Lemma \ref{lema-toda-env} for the top rectangle which lies in 
diagram (PB') to obtain 
$a\in \langle \CJ^\perp, \CI^\perp\rangle_\frakE$. Since $a$ is an $(\CI\CJ)^\perp$-preenvelope 
we obtain $(\CI\CJ)^\perp\subseteq \langle \CJ^\perp, \CI^\perp\rangle_\frakE$. 
\end{proof}

{
We have a converse for Corollary \ref{perp-de-idempotent}.
}

\begin{corollary}\label{cor-idempotent}
Suppose that there are enough $\frakE$-injective homomorphisms. 
A special precovering ideal $\CI$ in $\CA$ is idempotent (i.e. $\CI^2=\CI$) if and only if  
$\CI^\perp$ is closed with respect to Toda brackets. 
\end{corollary}

\begin{proof}
If $\CI$ is idempotent then $\langle \CI^\perp, \CI^\perp\rangle_\frakE=(\CI\CI)^\perp=\CI^\perp$.

Conversely, from $\CI^\perp\subseteq (\CI^2)^\perp=\langle \CI^\perp,
\CI^\perp\rangle_\frakE\subseteq \CI^\perp$ 
it follows that $\CI^\perp=(\CI^2)^\perp$. By Theorem \ref{cotorsion-precovering}, 
using the fact that both ideal $\CI$ and $\CI^2$ are special precovering, 
we have $\CI^2={^\perp((\CI^2)^\perp)}={^\perp(\CI^\perp)}=\CI$.
\end{proof}

{%

As in the study of ideal cotorsion pairs in exact categories, we can state the following version of (co-)Ghost Lemma. 


\begin{corollary}\label{ghost-complete}
Suppose that there are enough $\frakE$-injective and $\frakE$-projective homomorphisms. Then 
\begin{enumerate}[{\rm (1)}]
\item the class of special precovering ideals is closed with respect to products and
Toda brackets;

\item
the class of special preenveloping ideals is closed with respect to products and
Toda brackets;

\item If $(\CI,\CJ)$ and $(\CK,\CL)$ are two complete ideal cotorsion pairs then 
\begin{enumerate}[{\rm (a)}]
\item $(\CI\CK)^\perp=\langle \CL,\CJ\rangle_\frakE$ and $\langle \CI,\CK\rangle_\frakE={^\perp}{(\CL\CJ)}$;
\item $\langle \CI,\CK\rangle_\frakE^\perp=\CL\CJ$ and $\CI\CK={^\perp}{\langle \CL,\CJ\rangle_\frakE}$
\end{enumerate}
\end{enumerate}
\end{corollary}

\begin{remark}\label{assoc-toda}
From the above result it follows that if we have enough $\frakE$-injective and $\frakE$-projective homomorphisms 
the Toda bracket operation is associative on the class of special 
precovering (resp. preenveloping) ideals. In particular, we obtain that the Toda bracket operation computed with respect to 
the class $\frakD$ of all triangles from $\CT$ is associative 
for precovering (resp. preenveloping) ideals.  This is used in \ref{full-funct-sect}. 
In the case of ideals in exact categories the associativity is proved in the general setting in 
\cite[Proposition 8]{Fu-Herzog}. This is also valid for object ideals in triangulated category (as a consequence of 
\cite[Lemma 1.3.10]{perverse}). We are not able to prove decide if this property is valid for arbitrary ideals.
\end{remark}





\section{Ideal cotorsion pairs and relative phantom ideals}\label{Section-relative-cotorsion-pairs}

In this section we extend the ideal cotorsion theory introduced in \cite{Fu-et-al} to triangulated categories.
In order to do this \textsl{we fix a triangulated category $\CT$, a full subcategory $\CA$ which is closed under extensions,
and an almost exact structure $\frakE$ in $\CA$.}

\subsection{Relative phantom ideals} 
Given an almost exact structure $\frakF\subseteq\frakE$ we will construct the ideal of  phantom maps of $\frakF$ relative to $\frakE$. 
This is a generalizatiom of Herzog's construction of phantoms with respect to pure exact sequences, \cite{Herzog-ph}. 

\begin{definition} Let $\frakF$ be an almost exact structure in $\CA$ such that $\frakF\subseteq\frakE$. 
A map $\phi:X\to A$ from $\CA$ is called \textsl{relative $\frakF$-phantom} (\textsl{\wrt $\frakE$}), if $h\phi\in\Ph(\frakF)$, 
whenever $h\in\Ph(\frakE) $. We denote by $$\PEF=\{\phi\mid h\phi\in\Ph(\frakF)\hbox{ for all }h\in\Ph(\frakE)\}$$ 
the class of all relative
$\frakF$-phantom \wrt $\frakE$. 

Dually, a map $\psi:A\to X$ from $\CA$ is called \textsl{relative $\frakF$-cophantom} (\textsl{\wrt $\frakE$}), if $\psi h\in\Ph(\frakF)[-1]$, 
whenever $h\in\Ph(\frakE)[-1] $. We denote by $$\CPEF=\{\psi\mid \psi h\in\Ph(\frakF)[-1]\hbox{ for all }h\in\Ph(\frakE)[-1]\}$$ 
the class of all relative
$\frakF$-cophantom \wrt $\frakE$.
\end{definition}

 The proof of the following lemma is straightforward.

\begin{lemma}
If $\frakF$ is an almost exact structure and $\frakF\subseteq \frakE$ then $\PEF$ and $\CPEF$ are ideals in $\CA$.  
\end{lemma}

\begin{remark}
a) Informally a map $\phi:X\to A$ belongs to $\PEF$ if and only if for every base change  along $\phi$ 
of a triangle in $\frakE$,
\[ \xymatrix{ Y \ar[r]\ar@{=}[d]  & Z\ar[r]\ar[d]  & X \ar[r]\ar[d]^{\phi} & Y[1]\ar@{=}[d]  \\
 Y \ar[r]& C\ar[r]  & A \ar[r]^{h} & Y[1] 
 ,} \]
the top triangle is in $\frakF$.

b) Dually, a map $\psi:X\to A$ belongs to $\CPEF$ if and only if for every cobase change along $\psi$ 
of a triangle in $\frakE$,
\[ \xymatrix{  X\ar[r]\ar[d]^{\psi}  & Y \ar[r]\ar[d] & Z\ar@{=}[d]\ar[r]^{h[1]} & X[1]\ar[d]^{\psi[1]}  \\
 A \ar[r]& C\ar[r]  & Z \ar[r] & A[1] 
 ,} \] the bottom triangle is in $\frakF$.
\end{remark}

Remark that relative $\frakF$-phantoms and cophantoms \wrt $\frakE$ are ideals in $\CA$ whereas $\Ph(\frakE)$ is a 
phantom $\CA$-ideal.

However, we may always see $\Ph(\frakE)$ as a particular case of a relative phantom ideal. Indeed, if $\CA=\CT$ there is no difference between phantom $\CA$-ideals and ideals in $\CA$ and 
$\Ph(\frakE)=\PEF$, where $\frakD$ is the exact structure consisting of the class of all triangles in $\CT$. This explain the consistency of the notation of both by the same greek letter.

\begin{example}
a) Let $\CA$ be an exact idemsplit category, and embed it in a triangulated category $\CT=\mathbf{D}(\CA')$, where 
$\CA'$ is an abelian category containing $\CA$ as an extension closed subcategory (see Example \ref{exact-in-triang}). Then the class of all conflations in $\CA$ yields to an almost exact structure in $\CT$, denoted 
$\frakE$. If we consider a substructure of $\frakE$ then conflations in this substructure are short exact sequences in $\CA'$, so they also lead to an exact structure in $\CA$, denoted by $\frakF$. Then $\PEF$ and 
$\CPEF$ are exactly the class of phantom respectively cophantom maps considered in \cite{Fu-et-al} and \cite{Fu-Herzog}. 

b) Let $R$ be a ring and let $\CA$ be the category of all right $R$-modules. We view $\CA$ as a subcategory of the derived category of $\Modr R$. If $\frakE$ is the class of all exact sequences in $\CA$, and $\frakF$ is the class of all pure exact sequences in $\Modr R$ then $\PEF$ is the ideal studied by Herzog in \cite{Herzog-ph}.

c) In the case when $\frakF$ is the class of all splitting triangles, that is $\frakF=\frakD_0$, 
we have 
$\Ph(\frakD_0)=0$, so ${\Phi}_\frakE(\frakD_0)=\frakE\textrm{-Proj}$. 
Dually ${\Phi}^\frakE(\frakD_0)=\frakE\textrm{-Inj}$. 

d) If $\frakE=\frakD$ is the class of all triangles in $\CT$ then  ${\Phi}_\frakD(\frakF)=\Ph(\frakF)=
{\Phi}^\frakD(\frakF)$.

e) For a derived version of Herzog's phantoms, we consider  a ring $R$, and in the derived category $\mathbf{D}(\Modr R)$ we take as in \cite{Hovey-flat} the almost exact structure $\frakG$ defined as in example \ref{fantome-clasice-gen} by the condition that the all objects from the class 
$$\mathcal{P}_f=\{P[i]\mid P\in \Modr R \textrm{ is a finitely generated projective  module}\}$$ are $\frakG$-projective.  The $\frakG$-phantoms are called \textsl{ghosts}. It is obvious that the class $\frakF$ of all pure triangles is contained in $\frakG$. By \cite[Lemma 8.1 and Theorem 8.6]{Bel2000} we observe that $\frakG$ and $\frakF$ have enough projective and injective objects.

f) We also refer to \cite{Ch-98} for the topological versions of ghost and phantoms. In the stable homotopy theory the $\frakG$-phantoms of the exact structure $\frakG$ which is defined by the fact that all direct summand of direct sums of spheres are projective are called ghosts. Then the class of all pure triangles $\frakF$ is contaned in $\frakG$ and both these classes have enough projectives and injectives. In \cite[Section 6]{Ch-98} there are other examples of almost exact structures (contained in $\frakF$) defined by some projectivity conditions (the phantoms of these structures are called skeletal phantoms, respectively superphantoms). 
\end{example}

\begin{theorem}\label{orth-F}
Let $\frakF\subseteq \frakE$ be an almost exact structures. \begin{enumerate}[{\rm (1)}]
\item \begin{enumerate}[{\rm (a)}] \item 
The pair $(\Phi_\frakE(\frakF), \frakF\textrm{-}\mathrm{inj})$ is orthogonal.

\item If there are enough $\frakF$-injective homomorphisms then ${^{\perp} \frakF\textrm{-}\mathrm{inj}}= 
\PEF$.
\end{enumerate}
\item \begin{enumerate}[{\rm (a)}] \item 
The pair $(\frakF\textrm{-}\mathrm{proj}, \CPEF)$ is orthogonal.

\item If there are enough $\frakF$-projective homomorphisms then ${ \frakF\textrm{-}\mathrm{proj}^{\perp}}= 
\CPEF$.
\end{enumerate}
\end{enumerate}
\end{theorem}

\begin{proof}


(a) Let $e:B\to Y$ be an $\frakF$-injective homomorphism, $f:X\to A\in\Phi_\frakE(\frakF)$, and $\varphi:A\to B[1]\in\Ph(\frakE)$. 
We have $\varphi f\in\Ph(\frakF)$, hence $e[1]\varphi f=0$ since $e$ is $\frakF$-injective. Then $f\perp e$.

(b) By (a), it is enough to show that 
${^{\perp} \frakF\textrm{-}\mathrm{inj}}\subseteq \Phi_\frakE(\frakF)$. 
In order to do this, let us consider $f:X\to A$ a map from ${^{\perp}\frakF\textrm{-}\mathrm{inj}}$ and 
$\varphi:A\to B[1]$ from $\Ph(\frakE)$. Let $e:B\to Y$ be an $\frakF$-injective $\frakF$-inflation.  


Complete $\varphi$ to a triangle  $B\to C\to A\overset{\varphi}\to B[1]$ in $\frakE$, and consider the base-cobase change 
diagram   
\[ \xymatrix{
 B\ar[r]\ar@{=}[d] & C \ar[r] & A\ar[r]^{\varphi} & B[1] \ar@{=}[d] \\
 B\ar[r]^{i}\ar[d]_{e} & Z\ar[d]\ar[r]\ar[u]\ar@{-->}[dl]_z  & X\ar@{=}[d]\ar[r]^{\varphi f}\ar[u]^{f}  & B[1]\ar[d]^{e[1]}
\\
	        Y\ar[r] \ar@/_/ @{<--}[r]     & Z'\ar[r]      & X\ar[r]_0           & Y[1].
	       } \]
Since $e[1]\varphi f=0$, it follows that the triangle $Y\to Z'\to X\to Y[1]$ splits. 
Therefore there exists $z:Z\to Y$ such that $e=zi$.
Since $e$ is an $\frakF$-inflation, by Lemma \ref{infl-defl-factors} it follows that $i$ is an $\frakF$-inflation, 
hence $\varphi f\in\Ph(\frakF)$. 
Therefore, $\varphi f\in\Ph(\frakF)$ for all $\varphi\in\Ph(\frakE)$, hence $f\in \Phi_\frakE(\frakF)$.
\end{proof}


{
Moreover, for the case when $\frakF$ has enough projective homomorphisms we can use the ideal $\frakF\proj$ to see when an ideal 
is contained in $\mathbf{\Phi}_\frakE(\frakF)$.

\begin{proposition}\label{fantome-prin-proiective}
Let $\frakF\subseteq \frakE$ be an almost exact structures and let $\CI$ be an ideal in $\CA$. 
\begin{enumerate}[{\rm (1)}]
\item If there exist enough $\frakF$-projective homomorphisms, the following are equivalent:
\begin{enumerate}[{\rm (a)}]
\item $\CI\subseteq \PEF$;
\item $\CI(\frakF\proj)\subseteq \frakE\proj$.
\end{enumerate}
\item If there exist enough $\frakF$-injective homomorphisms, the following are equivalent:
\begin{enumerate}[{\rm (a)}]
\item $\CI\subseteq\CPEF$;
\item $(\frakF\inj)\CI\subseteq \frakE\inj$.
\end{enumerate}
\end{enumerate}
\end{proposition}

\begin{proof}
(a)$\Rightarrow$(b) We have $\Ph(\frakE)\CI(\frakF\proj)\subseteq \Ph(\frakF)(\frakF\proj)=0$, 
hence $\CI(\frakF\proj)\subseteq \frakE\proj$.

(b)$\Rightarrow$(a) We have to prove that $\mathfrak{PB}_\frakE(\CI)\subseteq \frakF$. By Corollary \ref{cor-prop1-suficiente},
 it is enough to prove that $\frakF\proj\subseteq \mathfrak{PB}_\frakE(\CI)\proj$. In order to obtain this, let us observe that 
$${\PB}_\frakE(\CI)(\frakF\proj)= \Ph(\frakE) \CI (\frakF\proj)\subseteq \Ph(\frakE) (\frakE\proj)=0,$$ and the proof is complete.  
\end{proof}

The following result shows us that $\frakE$-projective $\frakE$-deflations (resp. $\frakE$-injective $\frakE$-inlations) 
are test maps for relative $\frakF$-phantoms (relative $\frakF$-cophantoms).

\begin{proposition}\label{rel-phantom-vs-proj}
{\rm (1)} Let $K\to P\overset{p}\to A\overset{\psi}\to K[1]$ be a triangle in $\frakE$ such that $p$ is $\frakE$-projective.
A homomorphism $\varphi:X\to A$ is a relative $\frakF$-phantom \wrt $\frakE$ (i.e. $\varphi\in\Phi_\frakE(\frakF)$) 
if and only if $\psi\varphi\in\Ph(\frakF)$.

{\rm (2)} Dually, let $ A\overset{e}\to E\to L\overset{\psi}\to A[1]$ be a triangle in $\frakE$ such that 
$e$ is $\frakE$-injective.
A homomorphism $\varphi:A\to Y$ is a relative $\frakF$-cophantom \wrt $\frakE$ if and only if $\varphi[1]\psi\in\Ph(\frakF)$.
\end{proposition}

\begin{proof}
Suppose that $\psi\varphi\in\Ph(\frakF)$. 
We have to show that  $\zeta\varphi\in \Ph(\frakF)$ for every homomorphism $\zeta:A\to B[1]$ from $\Ph(\frakE)$. 

Let $\zeta:A\to B[1]$ be a homomorphism from $\Ph(\frakE)$.
Since $p$ is $\frakE$-projective, we have $\zeta p=0$,
hence there exists a map $g[1]:K[1]\to B[1]$ such that $g[1]\psi=\zeta$.  Moreover, we have $\psi\varphi\in\Ph(\frakF)$,
and it follows that $\zeta\varphi=g[1]\psi\varphi\in\Ph(\frakF)$, hence $\varphi\in\Phi_\frakE(\frakF)$.
 
Conversely, if $\varphi\in \Phi_\frakE(\frakF)$ then we apply the definition of $\Phi_\frakE(\frakF)$ to obtain that $\psi\varphi\in\Ph(\frakF)$.
\end{proof}

}



\subsection{The pullback and pushout almost exact stucture} In this subsection we intend to develop tools for  going back, from ideals to almost exact structures.
We start with a lemma which shows how to obtain almost exact sequences out a given ideal, using pushots and pulbacks:

\begin{lemma}\label{pull-push-ideal}
{\rm (1)} If $\CI$ is an ideal in $\CA$, then $\Ph(\frakE)\CI$ is a phantom $\CA$-ideal, 
whose corresponding almost exect structure consists of all triangles obtained as homotopy pullbacks 
of triangles in $\frakE$ along maps from $\CI$. 

{\rm (2)} If $\CJ$ is an ideal in $\CA$, then $\CJ[1]\Ph(\frakE)$ is a phantom $\CA$-ideal, 
whose corresponding almost exect structure consists of all triangles obtained as homotopy pushouts 
of triangles in $\frakE$ along maps from $\CI$.
\end{lemma}

\begin{proof}
(1) It is easy to see that $\alpha[1]\Ph(\frakE)\CI\beta\subseteq\Ph(\frakE)\CI$ for all $\alpha,\beta\in\CA^{\to}$. For showing that $\Ph(\frakE)\CI$ is a phantom $\CA$-ideal it is enough to prove that it is closed with respect to finite direct sums 
of homomorphisms. But this is true since both $\Ph(\frakE)$ and $\CI$ are closed with respect to finite direct sums. Now the statement concerning the almost exact structure is obvious.

The proof for (2) can be done in the same way.
\end{proof}

The \textsl{pulback almost exact structure associated to an ideal $\CI$} in $\CA$ is  class $\mathfrak{PB}_\frakE(\CI)$ of all
triangles  obtained as homotopy pullbacks 
of triangles in $\frakE$ along maps from $\CI$; the corresonding phantom $\CA$-ideal is $\PB(\CI)=\Ph(\frakE)\CI$. 

Dually, we define \textsl{the pushout almost exact structure associated to an ideal $\CJ$} of $\CA$  denoted $\mathfrak{PO}_\frakE(\CI)$  having the corresponding phantom $\CA$-ideal 
$\PO(\CJ)=\CJ[1]\Ph(\frakE)$. The subscript $\frakE$ may be removed if no danger of confusion occurs.

We will construct two Galois correspondences between ideals and almost exact structures 
in $\CA$. 

\begin{theorem}\label{Galois-correspondences}
Let  $\CT$ be a triangulated category. We fix a full subcategory $\CA$ which is closed under extensions,
and an almost exact structure $\frakE$ in $\CA$.
 The pairs of correspondences 
  $$\mathfrak{PB}_\frakE:\mathbf{Ideals}(\CA)\rightleftarrows \mathbf{Ex}(\frakE):{\Phi}_\frakE,$$
 respectively 
 $$\mathfrak{PO}_\frakE:\mathbf{Ideals}(\CA)\rightleftarrows \mathbf{Ex}(\frakE):{\Phi}^\frakE,$$
 between the class
 $\mathbf{Ideals}(\CA)$ of all ideals in $\CA$ and the class $\mathbf{Ex}(\frakE)$ of all an almost exact structures included in $\frakE$,
 determine two monotone Galois connections with respect to inclusion.
 \end{theorem}

\begin{proof}
Let $\CI$ be an ideal in $\CA$ and let $\frakF\subseteq \frakE$ be an almost exact structure. We have to prove that 
$\mathfrak{PB}_\frakE(\CI)\subseteq \frakF$ if and only if $\CI\subseteq\PEF$.

The inclusion
$$\CI\subseteq \PEF=\{\phi\mid h\phi\in\Ph(\frakF)\hbox{ for all }h\in\Ph(\frakE)\}$$
is equivalent to $\Ph(\frakE)\CI\subseteq \Ph(\frakF)$. Since $\Ph(\frakE)\CI=\PB(\CI)$, the last inclusion is equivalent 
to $\mathfrak{PB}_\frakE(\CI)\subseteq \frakF$.

The proof for the second pair is similar.
\end{proof}

Using the standard properties of Galois connections we have the following
\begin{corollary}\label{connection-PB-phantoms}
If $\frakF$ is an almost exact structure included in $\frakE$ and $\CI$ is an ideal in $\CA$ then: 
\begin{enumerate}[{\rm (1)}] \item $\mathfrak{PB}_\frakE(\PEF)\subseteq \frakF$ and 
$\CI\subseteq{\Phi}_\frakE(\mathfrak{PB}_\frakE(\CI))$;

\item $\mathfrak{PO}_\frakE(\CPEF))\subseteq \frakF$ and 
$\CI\subseteq{\Psi}_\frakE(\mathfrak{PO}_\frakE(\CI))$.
\end{enumerate}
\end{corollary}

The following results exhibit connections between orthogonal ideals and some injective/projective properties:  
   
\begin{proposition}\label{I-perp-ex-PB}
{\rm (1)} If $\CI$ is an ideal in $\CA$ then $\CI^\perp=\mathfrak{PB}_\frakE(\CI)\inj$.

{\rm (2)} If $\CJ$ is an ideal in $\CA$ then $^\perp \CJ=\mathfrak{PO}_\frakE(\CJ)\proj$. 
\end{proposition}

\begin{proof}
A homomorphism $j:A\to U$ is in $\CI^\perp$ if and only if $j[1]\Ph(\frakE)\CI=0$. 
But $\Ph(\frakE)\CI=\PB(\CI)$, and we apply 
Lemma \ref{basic-CI-inj} to obtain the conclusion.
\end{proof}

\begin{corollary}\label{PB-inclusions}
{\rm (1)} If $\CI$ is an ideal then $$\CI\subseteq \Phi_\frakE (\mathfrak{PB}_\frakE(\CI))\subseteq {^{\perp}(\CI^{\perp})},$$ both inclusions becoming equalities when $\CI$ is special precovering.


{\rm (2)} If $\CJ$ is an ideal then $$\CJ\subseteq \Phi^\frakE (\mathfrak{PO}_\frakE(\CJ))\subseteq ({^\perp\CJ})^{\perp},$$ both inclusions becoming equalities when $\CJ$ is special preenveloping.
\end{corollary}

\begin{proof} 
The first inclusion is 
a consequence of Corollary \ref{connection-PB-phantoms}. 


For the second inclusion, we replace in Theorem \ref{orth-F} the almost exact structure $\frakF$ by $\mathfrak{PB}(\CI)$,  
hence we have  $$\Phi_\frakE (\mathfrak{PB}_\frakE(\CI))\subseteq  {^\perp(\mathfrak{PB}(\CI)\textrm{-inj})}.$$
By Proposition \ref{I-perp-ex-PB} we have $\mathfrak{PB}(\CI)\textrm{-inj}=\CI^{\perp}$, hence 
$\Phi_\frakE (\mathfrak{PB}_\frakE(\CI))\subseteq {^{\perp}(\CI^{\perp})}.$ 
Finally if $\CI$ is special preevveloping the equality $\CI={^{\perp}(\CI^{\perp})}$ follows by Theorem \ref{cotorsion-precovering}.
\end{proof}

\begin{corollary} Let $\CI$ and $\CJ$ be ideals in $\CA$. \begin{enumerate}[{\rm (1)}]
\item If the phantom $\CA$-ideal $\PB(\CI)$ is precovering then  
$\Phi_\frakE (\mathfrak{PB}_\frakE(\CI))= {^{\perp}(\CI^{\perp})}$.

\item If the phantom $\CA$-ideal $\PO(\CJ)$ is preenveloping then  
$\Phi^\frakE (\mathfrak{PO}_\frakE(\CJ))= {({^{\perp}\CJ})^{\perp}}$.
\end{enumerate}
\end{corollary}

\begin{proof} 
By Theorem \ref{Th-procov-vs-Iinj} we obtain that 
there are enough $\mathfrak{PB}(\CI)$-injective homomorphisms. Using Proposition \ref{I-perp-ex-PB} and 
Theorem \ref{orth-F} we have  
$${^{\perp}(\CI^{\perp})}={^\perp (\mathfrak{PB}(\CI)\inj)}={\Phi}_\frakE(\mathfrak{PB}(\CI)),$$
and the proof is complete.
\end{proof}

\begin{corollary}\label{phi(F)-precovering}
{\rm (1)} Let us suppose that there are enough $\frakE$-projective homomorphisms and there are enough 
$\frakF$-injective homomorphisms. Then $\Phi_\frakE(\frakF)$ is a special precovering ideal.

{\rm (2)} Suppose that there are enough $\frakE$-injective homomorphisms and there are enough 
$\frakF$-projective homomorphisms. Then $\Psi_\frakE(\frakF)$ is a special preenveloping ideal.
\end{corollary}

\begin{proof}
By Theorem \ref{orth-F} we know that $\Phi_\frakE(\frakF)={^\perp \frakF\inj}$. But $\frakF\inj$ is a preenveloping ideal, 
hence we can apply Theorem \ref{salce-lemma} to obtain the conclusion.
\end{proof}

\subsection{Complete ideal cotorsion pairs}

In order to characterize the ideal cotorsion pairs which are complete, we will study first the existence of some special injective (projective) preenvelopes (precovers).

From the proof of Proposition \ref{subclasses-enough-inj} we can deduce that if $\frakH\subseteq \frakE$ are almost exact structures 
with enough injective homomorphisms then every $\frakH$-injective $\frakH$-inflation can be obtained as a pullback of an $\frakE$-triangle along
a suitable homomorphism from $\CA$. It is useful to consider some special $\frakH$-injective $\frakH$-inflations. 

An $\frakH$-injective homomorphism $e$ is \textsl{special \wrt $\frakE$} 
if it can be embedded in a homotopy pushout diagram 
\[ \xymatrix{& A\ar[r]^{e}\ar@{=}[d] & C\ar[d]\ar[r]& X\ar[d]^{\varphi}\ar[r] & A[1]\ar@{=}[d]\\
	      \mathfrak{d}: & A\ar[r]      & B\ar[r]      & Y\ar[r]           & A[1],} \]
such that $\mathfrak{d}\in\frakE$ and $\varphi\in\Phi_\frakE(\frakH)$. 
The notion of \textsl{special $\frakH$-projective homomorphism} is defined dually.

\begin{example}\label{exemple-special-inj}
From Corollary \ref{PB-inclusions} and Proposition \ref{I-perp-ex-PB} we observe that 
{ if $\CI$ is special precovering} then every special $\CI^\perp$-preenvelope is a special $\mathfrak{PB}_\frakE(\CI)$-injective homomorphism.
\end{example}

\begin{proposition}\label{special-injective}
Let $\frakH\subseteq \frakE$ be an almost exact structures. Then 
\begin{enumerate}[{\rm (1)}] 
\item Every special $\frakH$-injective homomorphism { is a special $\frakH\text{-{\rm inj}}$-preenvelope} and 
a special $\Phi_\frakE(\frakH)^{\perp}$-preenvelope.

\item Every special $\frakH$-projective homomorphism is a special $\frakH\text{-{\rm proj}}$-precover 
and a special $^\perp\Phi^\frakE(\frakH)$-precover.
\end{enumerate}
\end{proposition}

\begin{proof}
{Using Theorem \ref{orth-F} we observe that $\Phi_\frakE(\frakH)\subseteq{^\perp}\frakH\textrm{-inj}$,
 hence every special $\frakH$-injective homomorphism is a special $\frakH\textrm{-inj}$-preenvelope.}

Let $e$ be a special $\frakH$-injective homomorphism. 
By Corollary \ref{connection-PB-phantoms} we have the inclusion $\mathfrak{PB}_\frakE(\Phi_\frakE(\frakH))\subseteq \frakH$, 
hence we can apply 
Proposition \ref{I-perp-ex-PB} to obtain 
$$e\in \frakH\textrm{-inj}\subseteq \mathfrak{PB}_\frakE(\Phi_\frakE(\frakH))\textrm{-inj}=\Phi_\frakE(\frakH)^{\perp}.$$
Since $\Phi_\frakE(\frakH)\subseteq {^\perp (\Phi_\frakE(\frakH)^\perp)}$ we can apply the definition to obtain that $e$ is a 
special $\Phi_\frakE(\frakH)^\perp$-preenvelope.
\end{proof}

The following result improves Theorem \ref{orth-F}: 

\begin{theorem}\label{Iperp=injective} Let $\frakH\subseteq \frakE$ be an almost exact structures.
\begin{enumerate}[{\rm (1)}]
\item If there are enough special $\frakH$-injective homomorphisms then 
$$(\Phi_\frakE(\frakH),\frakH\textrm{-}\mathrm{inj})$$ is a  cotorsion pair 
{which is complete if $\frakE$ has enough projective homomorphisms.}

\item If there are enough special $\frakH$-projective homomorphisms then 
$$(\frakH\textrm{-}\mathrm{proj}, \Phi^\frakE(\frakH))$$ is an ideal cotorsion pair
{which is complete if $\frakE$ has enough injective homomorphisms.}
\end{enumerate}
\end{theorem}

\begin{proof}
Since we have enough special $\frakH$-injective homomorphisms, it follows that the ideal $\frakH\text{-inj}$ is 
a special preenveloping ideal and there are enough $\frakH$-injective homomorphisms. 
Then we can use Theorem \ref{orth-F} to obtain ${^{\perp} \frakH\textrm{-}\mathrm{inj}}= 
\Phi_\frakE(\frakH)$.
Now the conclusions are 
consequences of Theorem \ref{cotorsion-precovering}.
\end{proof}

We can characterize ideal cotorsion pairs in the case when we have enough $\frakE$-injective $\frakE$-inflations
and $\frakE$-projective $\frakE$-deflations. 

\begin{theorem}\label{mainthA} 
Let $\frakE$ be an almost exact structure such that there are enough $\frakE$-injective homomorphisms
and $\frakE$-projective homomorphisms, and let $(\CI,\CJ)$ be an ideal cotorsion pair. 

The following are equivalent:
\begin{enumerate}[{\rm (a)}]
 \item $\CI$ is precovering;
 \item $\CI$ is special precovering.
 \item $\CJ$ is preenveloping;
 \item $\CJ$ is special preenveloping.
\item There exists an almost exact structure $\frakH\subseteq \frakE$ with enough special injective homomorphisms such that $\CI=\mathrm{\Phi}_\frakE(\frakH)$;
\item There exists an almost exact structure $\frakF\subseteq \frakE$ with enough injective homomorphisms such that $\CI=\PEF$;
\item There exists an almost exact structure $\frakH\subseteq \frakE$ with enough special injective homomorphisms such that 
$\CJ=\frakH\inj$;
\item There exists an almost exact structure $\mathfrak{G}\subseteq \frakE$ with enough special projective homomorphisms such that $\CJ=\CPEG$;
\item There exists an almost exact structure $\mathfrak{F}\subseteq \frakE$ with enough projective homomorphisms such that $\CJ=\CPEF$;

\item There exists an almost exact structure $\mathfrak{H}\subseteq \frakE$ with enough special projective homomorphisms such that 
$\CI=\mathfrak{H}\proj$;
\end{enumerate}
\end{theorem}
 
\begin{proof}
The equivalences (a)$\Leftrightarrow$(b)$\Leftrightarrow$(c)$\Leftrightarrow$(d) are from Theorem \ref{salce-lemma}.  

The implications (f)$\Rightarrow$(b) and (h)$\Rightarrow$(d) are proved in Corollary \ref{phi(F)-precovering}. 
Finally, (b)$\Rightarrow$(e)
and (d)$\Rightarrow$(h) are in Corollary \ref{PB-inclusions} and Example \ref{exemple-special-inj}, while the equivalences (e)$\Leftrightarrow$(g) and (h)$\Leftrightarrow$(j) are obtained from Theorem \ref{Iperp=injective}.
\end{proof}

In the following example we will see that in general the almost exact structures $\frakF$ and $\frakH$ from the above theorem are not the same.

\begin{example}
Let $\CA$ the category $\Modr \Z$ of all abelian groups, $\frakE$ the class of all short exact sequences in $\Modr \Z$, and $\frakF$ the class of all pure exact sequences in $\Modr \Z$.

If $\CJ=\CPEF$, by Theorem \ref{orth-F} we have that $\frakF\proj\subseteq {^\perp\CPEF}$. It follows by Corollary \ref{Toda-perp-cor-1} that $\langle \frakE\proj,\frakF\proj\rangle_\frakE\subseteq \langle \frakE\proj, {^\perp\CPEF}\rangle_\frakE={^\perp\CPEF}$. If we suppose that in (j) we have $\mathfrak{H}=\frakF$ then we obtain that every extension of a pure projective abelian group by a projective abelian group is pure projective. 
However, it is easy to see that if we consider the group 
$$G=\left\{\frac{m}{n}\mid m,n\in\Z \textrm{ and } n \textrm{ is square free}\right\}\leq \Q$$ then $G/\Z\cong \oplus_{p\textrm{ is prime}}\Z/pZ$ is pure projective. But $G$ is not pure-projective, hence there exists there exists non pure-projective abelian groups which are extensions of pure-projectives by projectives. It follows that $\mathfrak{H}\neq \mathfrak{F}$.
\end{example}


\section{Applications}

\subsection{Projective classes}

We recall from \cite[Definition 2.5 and Proposition 2.6]{Ch-98} that \textsl{a projective class} in $\CT$ is a pair $(\CP,\CJ)$ where 
$\CP$ is a class of objects and $\CJ$ is an ideal in $\CT$ such that 
\[\CP=\{P\in\CT\mid \CT(P,\phi)=0\hbox{ for all }\phi\in\CJ\},\] 
\[\CJ=\{\phi\in\CT^\to\mid \CT(P,\phi)=0\hbox{ for all }P\in\CP\},\]
and every $X\in\CT$ lies in a triangle $P\to X\overset{\phi}\to Y\to P[1]$, with $P\in\CP$ and $\phi\in\CJ$.
As in \cite[Section 3]{Ch-98}, we consider the case when $\CP$ and $\CJ$ are suspension closed.

\begin{proposition}\label{cor-projective-class}
If $(\CP,\CJ)$ is a projective class such that $\CP$ and $\CJ$ are suspension closed and $\CI=\Ideal\CP$, then $(\CI,\CJ)$ is a complete ideal cotorsion pair with respect to $\frakD$.

Conversely, if $\CQ$ is a class of objects closed with respect to  direct sums such that $(\Ideal\CQ,\CJ)$ is a (special) cotorsion 
pair with respect to  $\frakD$ then $(\mathrm{add}(\CQ),\CJ)$ is a projective class.
\end{proposition}

\begin{proof}
The first statement follows from Example \ref{ex-trivial}.

For the second statement, let $\CQ$ be a class closed with respect to  finite direct sums. For every object $A$ we fix a special 
$\CJ$-preenvelope $j:A\to Y$. By Corollary \ref{salce-lemma-cor},  via the commutative 
diagram (constructed as in Example \ref{trivial-case-special})
\[\xymatrix{A[-1]\ar[d]^{j[-1]}\ar[r] & 0\ar[d]\ar[r]  & A\ar@{=}[d]\ar@{=}[r] & A\ar[d]^{j}\\
	    Y[-1]\ar[r]                & L\ar[r]^{i}& A\ar[r]^{j}         & Y,}
\] 
we obtain that the cocone $i:L\to A$ of $j$ is a special $\Ideal\CQ$-precover. Since 
$i$ factorizes through an object from $\CQ$, there exists $Q\in \CQ$ and a commutative diagram
\[\xymatrix{L\ar[d]\ar[r]^{i} & A\ar@{=}[d]\ar[r]^{j}  & Y\ar[d]\ar[r] & L[1]\ar[d]\\
	    Q\ar[r]                & A\ar[r]^{\varphi} & T\ar[r]         & Q[1].}
\] 
 Since $\varphi\in\CJ$, we apply the remarks stated in Example \ref{exemple-obiect-ideale} to conclude the proof.
\end{proof}

In particular we obtain Christensen's Ghost Lemma. 

\begin{example}
Let $(\CP,\CJ)$ and $(\CQ, \CL)$ be two projective classes in $\CT$. From Proposition \ref{cor-projective-class}
we know that $(\Ideal\CP, \CJ)$ and $(\Ideal\CQ,\CL)$ are complete ideal cotorsion pairs with respect to  the proper 
class $\frakD$ of all triangles in $\CT$. Then $(\langle \Ideal\CQ,\Ideal\CP\rangle_\frakD, \CJ\CL)$ is a complete ideal cotorsion pair.
By Proposition \ref{lem-toda-obiecte} we know that $\langle \Ideal\CQ,\Ideal\CP\rangle_\frakD=\Ideal\CV$, where $\CV$ is 
the class of all objects $V$ which lie in triangles $Q\to V\to P\to P[1]$ 
with $P\in\CP$ and $Q\in\CQ$. Applying again Proposition \ref{cor-projective-class}, it follows that 
$(\Ob(\Ideal\CV),\CJ\CL)$ is a projective class, and it is easy to see that $X\in \Ob(\Ideal\CV)$ if and only if $X$ is a direct summand 
of an object from $\CV$, hence $$(\mathrm{add}(\CV),\CJ\CL)$$ is a projective class.
\end{example}

\begin{remark}
Dually, we can consider injective classes $(\CI,\CQ)$, and the duals of above results are also valid. For 
the case when $\CT$ is a $k$-category ($k$ is a field) and the homomorphisms groups $\CT(A,B)$ are finitely dimensional for all objects 
$A$ and $B$ in $\CT$ then it is easy to see that for every object $A$ in $\CT$ the class $\aadd(A)$ is precovering and preenveloping. Therefore
it induces a projective class $(\aadd(A),\CJ)$ and an injective class $(\CI,\aadd(A))$. Here the homomorphisms from $\CJ$ (resp. $\CI$) are
 called $\CA$-ghosts (co-ghosts). A direct application of (co-)Ghost Lemma \ref{ghost-complete} and Proposition \ref{cor-projective-class}
lead us to the Ghost/Co-ghost Lemma and Converse proved in \cite[Lemma 2.17]{orlv-spc}.    
\end{remark}

Moreover, we have the following \textsl{dual of Christensen's Ghost Lemma}:

\begin{corollary}\label{clase-proj-telescope}
Let $(\CP,\CJ)$ and $(\CQ,\CK)$ be two projective classes in $\CT$, and denote by $\overline{\CT}(\CP,\CQ)$ the ideal of all homomorphisms which 
factorize through a homomorphism $P\to Q$ with $P\in \CP$ and $Q\in \CQ$. Then the pair 
$$(\overline{\CT}(\CP,\CQ),\langle \CJ,\CK\rangle_\frakD)$$
is an ideal cotorsion pair with respect to the class of all triangles in $\CT$. 
\end{corollary}

\begin{proof}
This follows from Corollary \ref{ghost-complete} since $\Ideal\CQ\Ideal\CP=\overline{\CT}(\CP,\CQ)$.
\end{proof}
}


\subsection{Krause's telescope theorem for projective classes}

We will apply the previous results to extend \cite[Proposition 4.6]{Krause-tel} to projective classes in compactly generated triangulated 
categories. In order to do this we will use the following

\begin{setup}\label{setup} Let $\CT$ be a compactly generated category and denote by $\CT_0$ a representative set of compact objects in $\CT$. Then this induces
 a projective class $(\CP, \mathcal{P}h)$, where $\CP=\Add(\CT_0)$ is the class of 
 pure-projective objects in $\CT$ and $\mathcal{P}h$ is the 
class of (classical) phantoms in $\CT$. We consider 
the almost exact structure $\frakF=\frakF_\CP$ of all pure triangles in $\CT$, i.e. 
$\Ph(\frakF)=\mathcal{P}h$.
We also fix   a projective class $(\CB,\CJ)$, and 
we make the following remarks and notations: 
\begin{enumerate}
                  \item 
                  the ideal $\overline{\CT}(\CB,\CP)$ of all 
homomorphisms which factorize through homomorphisms $B\to P$ with $B\in \CB$ and $P\in \CP$ 
is a precovering ideal; 
\item  as in Example 
\ref{ex-constr-suf-proj} we consider the weak proper 
class $\frakE(\CB)=\frakE_{\overline{\CT}(\CB,\CP)}$ of all triangles such that all elements in 
$\overline{\CT}(\CB,\CP)$ are projective with respect to these triangles; 
\item  
it is easy to see that
$\frakF\subseteq \frakE(\CB)$, so we can consider the class $\mathbf{\Phi}_{\frakE(\CB)}(\frakF)$ 
of all relative $\frakF$-phantoms
associated to $\frakE(\CB)$;
\item we denote 
by $\CI_\CB$ the set of all homomorphisms between compact objects which factorize through an object $B\in \CB$, and 
\item we consider the ideal
$$\overline{\CI}_\CB=\Ideal{\Add(\CT_0)}\Ideal\CB\Ideal{\Add(\CT_0)},$$ i.e. the ideal generated by class of all homomorphisms between 
pure-projective objects which factorize through an object $B\in \CB$. 
\end{enumerate}
\end{setup}

The relative $\frakF$-phantoms associated to $\frakE(\CB)$ can be characterized in the following way:

\begin{lemma}\label{lema-fantome-relative-pc}
The following are equivalent for a homomorphism $\varphi:A\to B$:
\begin{enumerate}[{\rm (1)}]
\item $\varphi\in \mathbf{\Phi}_{\frakE(\CB)}(\frakF)$;
\item for every compact object $C$ and every homomorphism $\alpha:C\to A$ we have 
$\varphi\alpha\in\overline{\CT}(\CB,\CP)$.
\end{enumerate}
\end{lemma}

\begin{proof}
(1)$\Rightarrow$(2) From Proposition \ref{fantome-prin-proiective} it follows that for every compact object $C$ and 
every homomorphism $\alpha:C\to A$ the homomorphism $\varphi\alpha$ is $\frakE(\CB)$-projective. Using  Example \ref{ex-constr-suf-proj} we obtain 
that $\varphi\alpha\in\overline{\CT}(\CB,\CP)$. 

(2)$\Rightarrow$(1) In order to apply Proposition \ref{fantome-prin-proiective}, we have to prove that 
for every pure-projective object $P$ and every homomorphism $\alpha:P\to A$ the homomorphism $\varphi\alpha$ is $\frakE$-projective. 
Since $P$ is a direct summand of a direct sum of compact objects, we can assume w.l.o.g. that $P=\oplus_{i\in I}C_i$ is a direct sum of compact objects. Then for every $i\in I$ we have $\varphi\alpha u_i$ is $\frakE(\CB)$-projective ($u_i$ denotes the canonical map 
$C_i\to\oplus_{i\in I}C_i$) and we apply Lemma \ref{direct-sum-inj-proj} to obtain the conclusion. 
\end{proof}

{
\begin{lemma}\label{elementar-small}
Let $\CT$ be an additive category. If $u:C\to A$, $v:A\to D_1\oplus D_2$ and $\alpha:C\to D_1$ are homomorphisms in $\CT$ such that 
$vu=\rho\alpha$ (i.e. $vu$ factorizes through $\rho$), where $\rho:D_1\to D_1\oplus D_2$ is the canonical homomorphism, then 
$\alpha=\pi v u$, where $\pi:D_1\oplus D_2\to D_1$ is the canonical projection (i.e. $\alpha$ factorizes through $A$).  
\end{lemma}

\begin{proof}
From $\rho\alpha=vu$ we obtain $\rho\pi vu=\rho\pi\rho\alpha=\rho\alpha$, hence $\pi vu=\alpha$ since $\rho$ is split mono.
\end{proof}
}

\begin{lemma}\label{lemma-prinB}
If $\CC$ is a category with direct sums, and $F:\CT\to \CC$ is a functor
which commutes with direct sums, the following are equivalent:
\begin{enumerate}[{\rm (a)}]
\item $F(\CI_\CB)=0$;
\item $F(\overline{\CI}_\CB)=0$.
\end{enumerate}
\end{lemma}

\begin{proof}
(a)$\Rightarrow$(b) {
It is enough to prove that if we consider two arbitrary families $(C_\lambda)_{\lambda\in \Lambda}$  and $(D_\kappa)_{\kappa\in K}$ of compact objects then
for every homomorphism 
$$\alpha:\oplus_{\lambda\in \Lambda}{C_\lambda}\to A \to  \oplus_{\kappa\in K}D_\kappa$$ with $A\in\CB$ we have $F(\alpha)=0$. 
If $\alpha:\oplus_{\lambda\in \Lambda}{C_\lambda}\to A\to \oplus_{\kappa\in K}D_\kappa$ then 
we observe that $$F(\alpha):\oplus_{\lambda\in \Lambda}{F(C_\lambda)}\to F(A)\to F(\oplus_{\kappa\in K}D_\kappa).$$
Since $F$ commutes with respect to direct sums, $F(\oplus_{\lambda\in \Lambda}{C_\lambda})$ is the direct sum of the family 
$(F(C_\lambda))_{\lambda\in\Lambda}$, and the canonical homomorphisms associated to this direct sum are $F(u_\lambda)$, $\lambda\in\Lambda$, where $u_\lambda$ are the canonical homomorphisms associated to the direct sum $\oplus_{\lambda\in \Lambda}{C_\lambda}$. 
Hence $F(\alpha)$ can be identified to a family 
$(F(\alpha_\lambda))_{\lambda\in\Lambda}$ of homomorphisms $\alpha_\lambda:C_\lambda\to A\to \oplus_{\kappa\in K}D_\kappa$. 
Since every $C_\lambda$ is compact, using Lemma \ref{elementar-small} we observe that every homomorphism $\alpha_\lambda$ can be viewed as a homomorphism 
$\alpha'_\lambda:C_\lambda\to A\to \oplus_{\kappa\in K_\lambda}D_\kappa$, where $K_\lambda$ are finite subsets of $K$ for all 
$\lambda\in\Lambda$. Since $F(\alpha'_\lambda)=0$ for all $\lambda$, it follows that $F(\alpha)=0$.  
}
\end{proof}

We recall that a covariant functor $H:\CT\to \CA$, where $\CA$ is abelian, is called \textsl{cohomological} if it sends triangles 
to exact sequences.

\begin{proposition}\label{krause-generalized}
Let $\CT$ be a compactly generated triangulated categories, and let  
$(\CB,\CJ)$ be a projective class in $\CT$.
The following are equivalent for a Grothendieck category $\CA$ and a cohomological functor $H:\CT\to \CA$ 
which commutes with direct sums:  
\begin{enumerate}[{\rm (a)}]
\item $H(\mathbf{\Phi}_{\frakE(\CB)}(\frakF))=0$;
\item $H(\CI_\CB)=0$.
\end{enumerate}
\end{proposition}

\begin{proof} (a)$\Rightarrow$(b) is obvious since $\CI_\CB\subseteq \frakE(\CB)\proj\subseteq \mathbf{\Phi}_{\frakE(\CB)}(\frakF)$.

(b)$\Rightarrow$(a) Let $\psi:X\to A$ be a homomorphism from $\mathbf{\Phi}_{\frakE(\CB)}(\frakF)$.

It is easy to see that 
$$\overline{\CI}_\CB=\Ideal{\CP}\Ideal{\CB}\Ideal{\CP}=\overline{\CT}(\CB,\CP)\Ideal{\CP}$$ 
is a precovering ideal, so the $\overline{\CI}_\CB$-orthogonal ideal with respect the class $\frakD$ of all triangles is 
$$\overline{\CI}_\CB^\perp=\langle \mathcal{P}h,\overline{\CT}(\CB,\CP)^\perp\rangle_\frakD, $$ 
and it is a (special) preenveloping ideal. 

Therefore every object $A$ from $\CT$ has a special $\overline{\CI}_\CB^\perp$-preenvelope $\gamma_A:A\to A^*$ which can be obtained as 
a composition $A\overset{\mu}\to Y\overset{\nu}\to A^*$ of two homomorphisms which lie in the solid part of the commutative diagram 
\[\tag{ENV}\xymatrix{        &   & C\ar@{-->}[d]^{\xi}\ar@{-->}[dddll] & \\
 &   & X\ar[d]^{\psi} & \\
    &   & A\ar[d]^{\alpha}\ar[dl]_{\mu} & \\
 X\ar[r]^{f}\ar[d]^{\varphi} & Y\ar[r]^{g}\ar[dl]^{\nu}    & Z\ar[r]  & X[1]\\
 A^*  &    &  & \ \ \ ,	}\] 
where $\alpha\in \overline{\CT}(\CB,\CP)^\perp$, $\varphi\in\mathcal{P}h$, and the horizontal line is a triangle in $\CT$.
Moreover, since $\gamma_A$ is a special $\overline{\CI}_\CB^\perp$-preenvelope, we have a commutative diagram 
\[\tag{SENV}\xymatrix{I[-1]\ar[r]\ar[d]^{i[-1]} & A\ar[r]^{\gamma_A}\ar@{=}[d] & A^*\ar[d]\ar[r]    & I\ar[r]\ar[d]^{i}   & A[1]\ar@{=}[d]\\
  U[-1]\ar[r] &A\ar[r]^{e}           & E\ar[r]    & U\ar[r]& A[1]          
}\] with $i\in \overline{\CI}_\CB$. By Lemma \ref{lemma-prinB}, since
$H(\CI_\CB)=0$, we obtain $H(\overline{\CI}_\CB)=0$, hence $H(\gamma_A)$ is a monomorphism. So, in order to obtain $H(\psi)=0$ it is enough to prove that 
$H(\gamma_A\psi)=0$.

Let $C$ be a compact object and $\xi:C\to X$ a homomorphism. By Lemma \ref{lema-fantome-relative-pc} we obtain that 
$\psi \xi\in \CT(\CB,\CP)$, hence 
$g\mu\psi \xi=\alpha\psi \xi=0$. Therefore $\mu\psi \xi$ factorizes through $f$, hence $\nu\mu\psi\xi$ factorizes through $\varphi$. But $\varphi$ is a 
phantom and $C$ is compact, and this implies $\nu\mu\psi\xi=0$. Then $\nu\mu\psi\in\mathcal{P}h$, hence $\gamma_A\psi$ is a phantom. 
By \cite[Corollary 2.5]{Krause-tel} we obtain $H(\gamma_A\psi)=0$, and the proof is complete. 
\end{proof}

Let $\CT$ and $\CC$ be compactly generated triangulated categories, and let $F:\CT\to \CC$ be a functor. We recall that $F$ is 
\textsl{a localizing functor} if it has a right adjoint $G:\CC\to \CT$ such that the induced natural transformation 
$FG\to 1_{\CC}$ is an isomorphism. 
Note that every localizing functor commutes with 
respect to direct sums. 
We apply Proposition \ref{krause-generalized} for the particular case when $\CB$ is a 
the kernel of a localizing functor. For further reference, let us remark that in this case $\CB$ is a \textsl{localizing subcategory}, 
i.e. it is a full triangulated subcategory of $\CT$ which is closed under direct sums. 

Let $F:\CT\to \CC$ a localization functor. If $G:\CC\to \CT$ is its right adjoint and $\eta:1_\CT\to GF$ is the induced 
natural transformation then for every $X\in \CT$ we can fix a triangle $$X'\overset{\nu_X}\to X\overset{\eta_X}\to GF(X)\to X'[1].$$ 
Applying $F$ we obtain that $F(\eta_X)$ is an isomorphism, and it follows that $F(X')=0$, hence $X'\in \CB$. 

Let $\CB=\{X\in\CT\mid F(X)=0\}$ be the kernel of $F$.
For every $B\in\CB$ we have $\CT(B,GF(X))\cong \CT(F(B),F(X))=0$, hence $\eta_X\in \CB^\perp$. Since $\CB$ is closed with respect to direct 
summands, we can apply \cite[Lemma 3.2]{Ch-98} to conclude that $(\CB,\CB^\perp)$ is a projective class (the orthogonal class $\CB^\perp$ is 
computed with respect to the class $\frakD$ of all triangles). In fact the ideal $\CB^\perp$ is in this case an object ideal. 

For every $B\in\CB$ the abelian group homomorphism $\Hom(B, \nu_X)$ is an isomorphism, 
hence $\nu_X$ is an $\Ideal{\CB}$-precover and $\eta_X$ is a $\CB^\perp$-preenvelope for all $X\in \CT$.


%

We have the following corollary.

\begin{corollary} \label{localizing-phantoms-obj}
Let $F:\CT\to \CC$ be a localizing functor between the compactly generated triangulated categories $\CT$ and $\CC$. 
If $\CB=\Ker(F)$ and we keep the notations used in this subsection then
$\Ob(\Phi_{\frakE(\CB)}(\frakF))\subseteq \CB$.
\end{corollary}

\begin{proof}
As in \cite{Krause-tel} we consider the Grothendieck category $\Modr \CC_0$ of all contravariant functors $\CC_0\to Ab$, 
and the functor $h_\CC:\CC\to \Modr \CC_0$ defined by $h_\CC(X)=\CC(-,X)_{|\CC_0}$. Then $h_\CC F:\CT\to Ab$ is a cohomological functor
such that $\Ker(h_\CC F)=\CB$. Then $h_\CC F(\CI_\CB)=0$, and it follows that $\Ob(\Phi_{\frakE(\CB)}(\frakF))\subseteq \Ker(h_\CC F)=\CB$.
\end{proof}

\subsection{Smashing subcategories}
Let $F$ be a localizing functor and $G$ its right adjoint. If $G$ also commutes with respect to direct sums then $F$ is \textsl{smashing}.
A triangulated subcategory $\CB$ of $\CT$ is a \textsl{smashing subcategory} if and only if there exists a smashing functor $F$ such that $\CB=\Ker(F)$.  
Note that a triangulated subcategory $\CB$ of $\CT$ is a smashing if and only if $\CB$ is a localizing subcategory of $\CT$ such that  every homomorphism 
$C\to B$ with $C\in\CT_0$ and $B\in B$ can be factorized as 
$$(C\to B)=(C\to B'\to C'\to B)$$ with $B'\in\CB$ and $C'\in \CT_0$, \cite[Theorem 4.2]{Krause-tel}. 

Therefore, the following corollary of Lemma \ref{lema-fantome-relative-pc} (using the notations from Setup \ref{setup}) gives us a version of \cite[Theorem 4.9]{Krause-tel}. 

\begin{corollary}\label{smashing-as-ideal}
A localizing subcategory $\CB$ is a smashing subcategory if and only if $\Ob(\Phi_{\frakE(\CB)}(\frakF))= \CB$. 
\end{corollary}

\begin{remark}
The above corollary says us that Proposition \ref{krause-generalized} is a generalization for \cite[Proposition 4.6]{Krause-tel}. 
\end{remark}

We will say that a smashing subcategory $\CB$ of a compactly generated triangulated category $\CT$ 
\textsl{satisfies the telescope conjecture} if for every compactly generated 
triangulated category $\CC$ and every exact functor $H:\CT\to \CC$  which preserves direct sums and annihilates  the subcategory $\CB_0$ of all compact objects $C\in \CB$ we obtain $H(\CB)=0$. 
Note that this is equivalent to the fact that $\CB$ is the smallest 
smashing subcategory which contains $\CB_0$. 

In the following we will present a characterization for smashing subcategories which satify the telescope conjecture by using relative phantom ideals. In order to do this, let us consider the following basic example of smashing subcategories:

\begin{lemma}\label{setup2}
Let $\CL$ be a set of compact objects in $\CT$. If $\frakE_\CL$ is the almost exact structure induced by the condition that all objects from $\CL$ are projective (as in Example \ref{fantome-clasice-gen}), i.e. 
$$\Ph(\frakE_\CL)=\{\varphi\mid \CT(\CL,\varphi)=0\}.$$ Then:
\begin{enumerate}[{\rm (1)}]
\item $\frakE_\CL$ has enough projective homomorphisms and 
a homomorphism is $\frakE_\CL$-projective if and only if it factorizes through an object from $\Add(\CL)$;

\item 
 $\frakF\subseteq \frakE_\CL$;

\item if $\CL$ is a triangulated subcategory of $\CT_0$ then $\Ob(\Phi_{\frakE_\CL}(\frakF))$ is a smashing subcategory.
\end{enumerate}
\end{lemma}

\begin{proof}
(1) This follows from the definition of $\frakF_\CL$.

(2) Since all objects from $\CL$ are compact, we have $\Hom(\CL,Ph)=0$, hence  $\frakF\subseteq \frakF_\CL$.

(3) In order to prove that $\Ob(\Phi_{\frakE_\CL}(\frakF))$ is a triangulated subcategory, we consider a triangle 
$Y\overset{\alpha}\to X\overset{\beta}\to Z\to Y[1]$ in $\CT$ such that $Y,Z\in \Ob(\Phi_{\frakE_\CL}(\frakF))$. Let $C$ be a compact object, and
$\gamma:C\to X$ a homomorphism. We have to prove that $\gamma$ is projective with respect to $\frakE_\CL$, i.e. $\gamma$ factorizes through 
an object from $\CL$. 

Since $Z\in \Ob(\Phi_{\frakE_\CL}(\frakF))$ and $C$ is compact we know that $\beta\gamma$ factorizes through an object $B\in\CL$. Therefore 
we have a commutative diagram   
\[ \xymatrix{  C'\ar[r]^{\zeta}\ar[d]^{\delta} & C\ar[d]^{\gamma}\ar[r]    & B\ar[r]\ar[d]   & C'[1]\ar[d]\\
  Y\ar[r]^{\alpha}           & X\ar[r]^{\beta}    & Z\ar[r]& Y[1]          
}\] such that the horizontal lines are triangles in $\CT$. Since $C'$ is compact and $Y\in \Ob(\Phi_{\frakE_\CL}(\frakF))$ the homomorphism 
$\delta$ factorizes through an object $B'\in \CL$, i.e. $\delta=\mu\nu$ with $\nu:C'\to B'$ and $\mu:B'\to Y$. 
Using a cobase change we can complete the above commutative diagram to 
the following commutative diagram  
\[ \xymatrix{ B'\ar[r]\ar@/_2pc/ [dd]_{\mu} & D\ar[r]  \ar@/_2pc/ @{-->}[dd]  & B\ar[r] 
& B'[1] \ar@/^2pc/[dd]^{\mu[1]}\\
 C'\ar[r]^{\zeta}\ar[u]^{\nu}\ar[d]_{\delta} & C\ar[d]^{\gamma}\ar[r]\ar[u]    
& B\ar[r]\ar[d]\ar@{=}[u]   & C'[1]\ar[d]_{\delta[1]}\ar[u]^{\mu[1]}\\
  Y\ar[r]^{\alpha}           & X\ar[r]^{\beta}    & Z\ar[r]& Y[1],          
}\]
where the dotted arrow exists since $\gamma\zeta=\alpha\delta=\alpha\mu\nu$. It follows that $\gamma$ factorizes through 
$D$. Since $\CL$ is a triangulated subcategory, we obtain that $D\in\CL$, hence $X\in \Ob(\Phi_{\frakE_\CL}(\frakF))$. 

It is easy to see that $\Ob(\Phi_{\frakE_\CL}(\frakF))$ is closed under direct sums and shifts, hence $\Ob(\Phi_{\frakE_\CL}(\frakF))$ is a localizing subcategory 
in $\CT$. 

Moreover  if $X\in \Ob(\Phi_{\frakE_\CL}(\frakF))$ and $C$ is a compact object in $\CT$ then every homomorphism $C\to X$ factorizes through
 an object $B\in \CL$, and we can write $$(C\to X)=(C\to B\overset{=} \to B\to X).$$
Since $\CL\subseteq \Ob(\Phi_{\frakE_\CL}(\frakF))$ we can apply \cite[Theorem 4.2]{Krause-tel} to obtain that $ \Ob(\Phi_{\frakE_\CL}(\frakF))$ is a 
smashing subcategory.
\end{proof}

\begin{proposition}\label{propo-smashing-tel}\label{smashing-telescope}
Let $\CB$ be a localizing subcategory of $\CT$. The following are equivalent:
\begin{enumerate}[{\rm (1)}]
 \item $\CB$ is smashing, and it satisfies the telescope conjecture;
 \item there exists a triangulated subcategory $\CL$ of $\CT_0$ such that $\CB=\Ob(\Phi_{\frakE_\CL}(\frakF))$.
\end{enumerate}
\end{proposition}

\begin{proof}
(1)$\Rightarrow$(2) Let $\CB_0$ be the subcategory of all compact objects from $\CB$. Let us denote by $\CP_0$ the class $\Add(\CB_0)$. Hence $\CP_0\subseteq \CB$ and $\CP_0\subseteq \CP$, and we have 
$$\overline{\CT}(\CB,\CP)\supseteq \overline{\CT}(\CP_0,\CP_0)=\Ideal{\CP_0}.$$ 
It follows that $\frakF\subseteq \frakE(\CB)\subseteq \frakE_{\CB_0}$, hence 
$$\Phi_{\frakE_{\CB_0}}(\frakF)\subseteq \Phi_{\frakE(\CB)}(\frakF).$$
By Corollary \ref{localizing-phantoms-obj} it follows that $\CB_0\subseteq 
\Ob(\Phi_{\frakE_{\CB_0}}(\frakF))\subseteq \CB$. 
Since $\Ob(\Phi_{\frakE_{\CB_0}}(\frakF))\subseteq \CB$ is smashing, it follows that 
$\CB= \Ob(\Phi_{\frakE_{\CB_0}}(\frakF))$.


(2)$\Rightarrow$(1) 
Conversely, let $H:\CT\to \CC$ be an exact functor between compactly generated triangulated categories such that it preserves direct sums
and $H(\CL)=0$. 

Let $X$ be an object from $\CB=\Ob(\Phi_{\frakE_\CL}(\frakF))$. If we consider a triangle  
$$Y\to P\overset{p}\to X\overset{\varphi}\to Y[1]$$ such that $p$ an $\Add(L)$-precover for $X$ then we deduce from Proposition \ref{rel-phantom-vs-proj} that $\varphi\in Ph$. Therefore every morphism $C\to X$ with $C$ a compact object factorizes through an element from $\Add(\CL)$. Since $C$ is compact and $\CL$ is closed under finite direct sums, it follows that  every morphism $C\to X$ with $C$ a compact object factorizes through an element from $\CL$.

Therefore, all homomorphisms from $\CI_\CB$ factorize through objects from $\CL$, hence $H(\CI_\CB)=0$. Using the same notations as in the proof of Corollary \ref{localizing-phantoms-obj}, and applying 
Proposition \ref{krause-generalized}, we obtain $h_\CC H(\CB)=0$. Then $H(\CB)=0$, and the proof is complete. 
\end{proof}

As a corollary we obtain the following characterization, proved by H. Krause in \cite[Theorem 13.4]{Krause-quo}. 

\begin{corollary}
Let $\CB$ be a smashing subcategory of a compactly generated subcategory $\CT$. The following are equivalent: 
\begin{enumerate}[{\rm (a)}]
\item $\CB$ satisfies the telescope conjecture; 
\item for every 
compact object $C$ and for every object $B\in \CB$ every homomorphism $C\to B$ factorizes through an object from $\CB\cap \CT_0$;
\item $\CB$ is a compactly generated as a triangulated category.
\end{enumerate}
\end{corollary}

\begin{proof}
(a)$\Leftrightarrow$(b) follows from Proposition \ref{propo-smashing-tel}.

(b)$\Rightarrow$(c) For every non-zero object $B\in \CB$ we can find a non-zero homomorphism $C\to B$ with $C$ a compact object in $\CT$. 
Applying (b) this homomorphism factorizes through a (non-zero) homomorphism $C_0\to B$ with $C_0\in \CB\cap \CT_0$.

(c)$\Rightarrow$(a) Let $F$ be a localizing functor which induces $\CB$, and let $G$ be its right adjoint.
As before, we denote by $\eta:1_\CT\to GF$ is the induced natural transformation. 

We first observe that if $B$ is a compact from $\CB$ and $\alpha:B\to X=\bigoplus_{i\in I}X_i$ is a homomorphism in $\CT$ 
then $\eta_{X}\alpha=0$. If we embed every object  $Y$ in the canonical diagram 
$$Y'\overset{\nu_Y}\to Y\overset{\eta_Y}\to GF(Y)\to Y'[1],$$ we observe that $X'=\bigoplus_{i\in I}X'_i\in \CB$, 
$\nu_{X}=\oplus_{i\in I}\nu_{X_i}$, and $\alpha$ factorizes through $\nu_X$. Since every homomorphism $B\to \bigoplus_{i\in I}X'_i$ factorizes
through a finite subset of $I$, it is easy to see that $\alpha$ has the same property. Therefore every compact from $\CB$ is compact in $\CT$
.

Since every object from $\CB$ is a homotopy colimit of pure-projective objects 
(cf. the proof of \cite[Theorem 3.1 and Lemma 3.2]{nee-jams}) and 
the homotopy colimits are computed in the same way in $\CB$ as in $\CT$ (as cones of Milnor's triangles), we can 
apply \cite[Lemma 2.8]{nee-jams} to obtain the conclusion.     
\end{proof}


\subsection{Full functors}\label{full-funct-sect} 
For Setup \ref{setup} we can apply Corollary \ref{ghost-complete} (Ghost Lemma) 
and Remark \ref{assoc-toda} to compute the right $\overline{\CI}_\CB$-orthogonal ideal (with
respect to the class $\frakD$ of all triangles) 
 $$\overline{\CI}_\CB^\perp=\langle \mathcal{P}h,\CJ,\mathcal{P}h\rangle_\frakD.$$
Therefore for every object $A$ from $\CT$ the $\overline{\CI}_\CB^\perp$-preenvelope $\gamma_A:A\to A^*$ can be obtained as 
a composition $A\overset{\mu}\to V\overset{\nu}\to A^*$ of two homomorphisms which lie in a commutative diagram 

\[\tag{ENV'} \xymatrix{    & &  &   & A\ar[d]^{\varphi}\ar[dl]\ar[ddlll] & \\
& & X\ar[r]\ar[d]^{\beta} & Y\ar[r]\ar[dl]    & Z\ar[r]  & X[1]\\
 U\ar[d]^{\psi}\ar[r] & V\ar[r]\ar[dl] & W\ar[r] & U[1] &   & \\ 
  A^*  & &  &    &  & \\
	}\] 
	such that $\varphi,\psi\in \mathcal{P}h$ and $\beta\in \Ideal{\CB}^\perp$.

 We will use this diagram to prove that in the case of full functors the hypothesis $F(\CI_\CB)=0$ always implies $F(\CB)=0$.

\begin{proposition}\label{H-full}
Let $\CT$ be a compactly generated triangulated category and $B$ an object in $\CT$. We denote 
by $\CI_B$ the set of all homomorphisms between compact objects which factorize through $B$. 
 If $\CA$ is a Grothendieck category and $F:\CT\to \CA$ a full cohomological functor which commutes with direct sums, 
the following are equivalent:
\begin{enumerate}[{\rm (a)}]
\item $F(B[n])=0$ for all $n\in\Z$;
\item $F(\CI_B[n])=0$ for all $n\in\Z$.
\end{enumerate}  
\end{proposition}
\begin{proof}
We will work with the projective class $(\Add(B), \CJ)$, and we consider the ideal $\overline{\CI}=\overline{\CI}_{\Add(B)}$.

Then for an object $A$ the $\overline{\CI}^\perp$-preenvelope 
$\gamma_A:A\to A^*$ can be embedded in a commutative diagram (ENV')
	such that $\varphi,\psi\in \mathcal{P}h$ and $\beta\in \Ideal{\Add(B)}^\perp$.

Using \cite[Corollary 2.5]{Krause-tel} we obtain that $F(\varphi)=0$ and $F(\psi)=0$. Then applying $F$ to the above diagram we obtain
the solid part of the following commutative diagram:
\[\xymatrix{    & &  &   & F(A)\ar[d]^{0}\ar[dl]\ar[ddlll]\ar@/_/@{-->}[dll]\ar@/_2pc/ @{-->}[ddllll] & \\
& & F(X)\ar[r]\ar[d]^{F(\beta)} & F(Y)\ar[r]\ar[dl]    & F(Z)\ar[r]  & F(X[1])\\
 F(U)\ar[d]^{0}\ar[r] & F(V)\ar[r]\ar[dl] & F(W)\ar[r] & F(U[1]) &   & \\ 
  F(A^*)  & &  &    &  & \ \ \ .\\
	}\] 
	Since the horizontal lines are exact sequences then we can complete the diagram with the homomorphism ${F(A)\dashrightarrow F(X)}$. But $F$ 
	is full, hence we can find a homomorphism $A\to X$ such that $F(A\to X)={F(A)\dashrightarrow F(X)}$. 

If $A\in \Add(B)$ then we have $$F(A)\dashrightarrow F(X)\to F(W)=F(A\to X\overset{\beta}\to W)=0$$
since $\beta\in\Ideal{\Add(B)}^\perp$, hence we can find the homomorphism 
$F(A)\dashrightarrow F(U)$. It follows that $F(\gamma_A)=0$.
Since $\gamma_A$ is a special preenveloping, as in the proof of Proposition \ref{krause-generalized}, we have a commutative diagram 
\[\xymatrix{I[-1]\ar[r]\ar[d]^{i[-1]} & A\ar[r]^{\gamma_A}\ar@{=}[d] & A^*\ar[d]\ar[r]    & I\ar[r]\ar[d]^{i}   & A[1]\ar@{=}[d]\\
  X[-1]\ar[r] &A\ar[r]^{e}           & E\ar[r]    & X\ar[r]& A[1]          
}\]
with $i\in \CI$. Since $F(i[-1])=0$, applying $F$ we obtain the exact sequence 
$$F(I[-1])\overset{0}\to F(A)\overset{0}\to F(A^*),$$
hence $F(A)=0$. 
\end{proof}

 \end{document}